\documentclass{amsart}

\usepackage{amsthm}
\usepackage{epic,eepic}
\usepackage[dvips]{graphicx}

\makeatletter

\newcounter{tempcounteri}
\newcounter{tempcounterii}
\newenvironment{enumeraterom}{%
  \begin{enumerate}%
   
   }{%
  \end{enumerate}}

\newcommand{\bmath}[1]{\mbox{\mathversion{bold}$#1$}}

\newcommand{\C}{\bmath{C}}

\newcommand{\Z}{\bmath{Z}}
\newcommand{\R}{\bmath{R}}

\newcommand{\CP}{\bmath{C\!P}}
\newcommand{\SL}{\operatorname{SL}}
\newcommand{\SU}{\operatorname{SU}}
\newcommand{\U}{\operatorname{U}}

\newcommand{\Herm}{\operatorname{Herm}}
\newcommand{\Trig}{\operatorname{{\mathcal T}}}
\newcommand{\cmcone}{\mbox{\rm CMC}-$1$}

\newcommand{\trace}{\operatorname{trace}}
\newcommand{\ord}{\operatorname{ord}}
\newcommand{\id}{\operatorname{id}}
\newcommand{\Hyp}{\mathcal{H}}
\newcommand{\metone}{\operatorname{Met}_1}
\renewcommand{\Re}{\operatorname{Re}}
\renewcommand{\Im}{\operatorname{Im}}
\newcommand{\Res}{\operatornamewithlimits{Res}}
   \newtheorem{theorem}{Theorem}[section]
   \newtheorem*{theorem*}{Theorem}
   \newtheorem{proposition}[theorem]{Proposition}
   \newtheorem{corollary}[theorem]{Corollary}
   \newtheorem{lemma}[theorem]{Lemma}
 \theoremstyle{definition}
   \newtheorem{definition}[theorem]{Definition}
   
 \theoremstyle{remark}
   \newtheorem{remark}[theorem]{Remark}
   \newtheorem*{remark*}{Remark}
\numberwithin{equation}{section}
\makeatother
\title[Symmetric CMC-1 surfaces with irregular ends]{
     Constructing mean curvature $1$ surfaces \\
            in $H^3$ with irregular ends 
   }

\author{Wayne Rossman}
\author{Masaaki Umehara}
\author{Kotaro Yamada}
\date{December 15, 2002}
\address[Rossman]{%
   Department of Mathematics, Faculty of Science,
   Kobe University,
   Rokko, Kobe 657-8501, Japan%
}
\email{wayne@math.kobe-u.ac.jp}
\address[Umehara]{%
   Department of Mathematics, Graduate School of Science,
   Osaka University,
   Osaka 560-0043, Japan%
}
\email{umehara@math.sci.osaka-u.ac.jp}
\address[Yamada]{%
   Faculty of Mathematics,
   Kyushu University 36, 6-10-1
   Hakozaki, Higashi-ku, Fukuoka 812-8581, Japan%
}
\email{kotaro@math.kyushu-u.ac.jp}
\begin{document}
\begin{abstract}
With the developments of the last decade on complete constant mean
curvature $1$ (\cmcone) surfaces in the hyperbolic $3$-space $H^3$, many
examples of such surfaces are now known.  
However, most of the known examples have regular ends.  
(An end is irregular, resp.~regular, if the hyperbolic Gauss map of the
 surface has an essential singularity, resp.~at most a pole, there.)  
There are some known surfaces with irregular ends, 
but they are all either reducible or of infinite total curvature.  
(The surface is reducible if and only if the monodromy of the 
secondary Gauss map can be simultaneously diagonalized.)  
Up to now there have been no known complete irreducible \cmcone{} surfaces 
in $H^3$ with finite total curvature and irregular ends.  

The purpose of this paper is to construct countably many $1$-parameter 
families of genus zero \cmcone{} surfaces with irregular ends and finite
total curvature, which have either dihedral or Platonic symmetries.  
For all the examples we produce, we show that they have finite total 
curvature and irregular ends.  
For the examples with dihedral symmetry and the simplest 
example with tetrahedral symmetry, we show irreducibility.  
Moreover, we construct a genus one \cmcone{} surface with four irregular
ends, which is the first known example with positive genus whose ends
are all irregular.  
\end{abstract}
\maketitle
\section*{Introduction}
Let $H^3$ denote the unique simply connected complete $3$-dimensional
Riemannian manifold with constant sectional curvature $-1$, which we
call the hyperbolic $3$-space.  
Associated to a complete finite-total-curvature \cmcone{}
(constant mean curvature one)
conformal immersion $f \colon{}M \to H^3$
of a Riemann surface $M$ are two meromorphic maps called the hyperbolic
Gauss map and the secondary Gauss map, which we denote by $G$ and $g$
respectively (to be defined in the next section).  
Using these two Gauss maps, we can define two characteristics of the
surface $f$:
\begin{enumerate}
\item It is known that $M$ is biholomorphic to a compact Riemann surface 
      with a finite number of points removed, and hence each end is
      conformally a punctured disk.  
      Therefore we may consider the order of the hyperbolic Gauss map
      $G$ at each end, and an end is called {\it regular\/} if $G$ has
      at most a pole singularity at this end.  
      If $G$ has an essential singularity, the end is called 
      {\it irregular}. 
\item Although $G$ is single-valued on $M$, the secondary Gauss map
      $g$ might be multi-valued on $M$, so we can have a nontrivial
      monodromy representation defined on the first fundamental group of
      $M$.  
      This monodromy group is a subgroup of $\SU(2)$, and if all members
      of this group can be diagonalized by the same conjugation, we say
      that the surface $f$ is {\it reducible}.  
      Otherwise, we say that $f$ is {\it irreducible}.  
      (Irreducibility depends on a global behavior of the surface but
      not on individual ends.)
\end{enumerate}
If a \cmcone{} immersion is reducible, the surface can be deformed
preserving its hyperbolic Gauss map $G$ and Hopf differential
($(2,0)$-part of the second fundamental form, see
Section~\ref{sec:prelim}).
On the other hand, an irreducible surface is the only surface with given
hyperbolic Gauss map and Hopf differential.

Recent progress in the theory of \cmcone{} surfaces in $H^3$ has led to
the discovery of many new examples of these surfaces.  
Many examples with regular ends are now known, and various properties of
these surfaces are understood. 
Bryant \cite{Bryant} found a local description for these surfaces in terms 
of holomorphic data that initiated this recent progress.  
The last two authors \cite{uy1}--\cite{uy7} developed the theory using
Bryant's description to find many examples and properties, and work in
this direction has been continued by Small \cite{Sm}, the authors 
\cite{ruy1}--\cite{ruy5}, Costa-Sousa Neto \cite{CN}, 
Earp-Toubiana \cite{ET1}--\cite{ET2}, Yu \cite{Yu1}--\cite{Yu3}, 
Levi-Rossman \cite{LR}, Barbosa-Berard \cite{BB}, do Carmo and Gomes and 
Lawson and Thorbergsson and Silveira \cite{CGT,CL,CS}, and others.  
Regarding properties of the ends of embedded examples, Collin, Hauswirth
and Rosenberg \cite{chr} have recently shown that any embedded \cmcone{}
surface of finite total curvature is either a horosphere or all of its
ends are asymptotic to catenoid cousin ends.  
In \cite{chr,Yu3}  it is further shown that any irregular end cannot
be embedded, and the limit points of such an end are dense at infinity.  
Recently, Pacard and Pimentel \cite{PP} established a method for
attaching small handles between tangent horospheres and deforming to
produce \cmcone{} surfaces, and this construction produces many embedded
\cmcone{} surfaces of any genus.  
Also, Karcher \cite{Kar} has recently constructed periodic \cmcone{}
surfaces with fundamental domains in several different types of compact
quotients of $H^3$.

A typical example of an irregular end is the end of the Enneper cousin,
a surface first constructed by Bryant \cite{Bryant}.  
After that the last two authors \cite{uy2} constructed examples of genus
zero and two irregular ends, and also many reducible \cmcone{} surfaces
of genus zero whose ends are all irregular, using deformations from
minimal surfaces.
(The conclusion of Remark 4.4 in \cite{uy2} contains an error.  
 The number of ends should be  $ml+2$, and hence the genus of 
 $\hat M_0^*$ is zero.) Recently, Daniel \cite{D} has investigated
  irregular ends
from the viewpoint of Nevanlinna theory.

After \cite{uy2}, no further surfaces with irregular ends and 
finite total curvature had been constructed. 
(However, such an example with infinite total curvature can be found in
\cite{ruy3}.)  
In particular, until now no irreducible \cmcone{} surfaces with
irregular ends and finite total curvature had been known.

The purpose of the paper is to construct countably many $1$-parameter 
families of genus zero \cmcone{} surfaces with irregular ends and finite
total curvature, which have either dihedral or Platonic symmetries.  
We further show that examples with dihedral symmetries, and the simplest 
example with tetrahedral symmetry, must be irreducible. 
 All of our examples have
irregular ends of finite type in the sense of Daniel \cite{D}.

To do the construction, we start with the meromorphic data for the genus 
zero irreducible \cmcone{} surfaces with regular ends found in
\cite{uy3} and \cite{ruy1} and modify this data to make surfaces with 
irregular ends.  
The spirit of the construction is similar to the construction of
trinoids in \cite{uy3}, where \cmcone{} surfaces with prescribed Gauss
maps are constructed by reflecting spherical triangles, and we use
monodromy killing arguments like in \cite{ruy1} and \cite{uy6}, but
the techniques are brought to bear more intricately here.  

In Section~\ref{sec:prelim} we give necessary preliminaries.
As our construction is done by reflecting abstract spherical triangles,
we discuss this in Section~\ref{sec:reflection}, and introduce a method
to construct \cmcone{} surfaces with irregular ends
(Theorem~\ref{thm:gen}), which is proved in
Section~\ref{sec:proof}.
As an application of the theorem, we construct examples of genus zero
with either dihedral or Platonic symmetries in Section~\ref{sec:platonic}.
Finally, in Section~\ref{sec:torus}, we construct a \cmcone{} surface of
genus $1$ with four irregular ends, which is the first known example
with positive genus whose ends are all irregular.
\section{Preliminaries}
\label{sec:prelim}
\subsection*{Null meromorphic curves.} 
Here we recall from \cite{uy1,uy7} some fundamental properties of null
meromorphic curves in $\SL(2,\C)$.
\begin{definition}
 Let $F\colon{}M \to \SL(2,\C)$ be a meromorphic map defined on a
 Riemann surface $M$ with a local complex coordinate $z$.  
 Then $F$ is called {\it null\/} if  $\det(F_z) \equiv 0$.  
 (This condition does not depend on the choice of coordinate $z$.) 
\end{definition}
Let $F\colon{}M \to \SL(2,\C)$ be a null meromorphic map. 
We define a matrix $\alpha$ by
\[
   \alpha=\begin{pmatrix} \alpha_{11} & \alpha_{12} \\
                        \alpha_{21} & \alpha_{22} 
                        \end{pmatrix}
         :=F^{-1}dF \; ,
\]
and set
\begin{equation}\label{eq:gomega}
   g:= \alpha_{11}/\alpha_{21}\;,\qquad
  \omega:=\alpha_{21} \; .
\end{equation}
Then the pair $(g,\omega)$ is a meromorphic function $g$ and a
meromorphic $1$-form $\omega$ on $M$ satisfying 
\begin{equation}\label{eq:ode0}
   F^{-1}dF= \begin{pmatrix} 
	      g & -g^2 \\
	      1 & -g\hphantom{^2}
	     \end{pmatrix} \, \omega \; .
\end{equation}
Conversely, let $g$ be a meromorphic function and $\omega$ a holomorphic
$1$-form on $M$. 
Then the ordinary differential equation \eqref{eq:ode0} is integrable
and the solution $F$ is a null map into $\SL(2,\C)$ (since we will always 
choose the initial condition to be in $\SL(2,\C)$) defined on the 
universal covering of $M\setminus\{\text{poles of $\alpha$}\}$.
In general, $F$ might not be single-valued on $M$ itself, and $F$ may have
essential singularities at poles of $\alpha$.  
We call the pair $(g,\omega)$ the {\it Weierstrass data\/} of $F$. 
\begin{definition}
 Let 
 \[
     F=\begin{pmatrix}F_{11} & F_{12} \\ F_{21} & F_{22}
     \end{pmatrix}
 \]
 be a null meromorphic map of $M$ into $\SL(2,\C)$ with 
 Weierstrass data $(g,\omega)$.
 We call 
 \[
      G:=\frac {dF_{11}}{dF_{21}}=\frac {dF_{12}}{dF_{22}}
 \]
 the {\it hyperbolic Gauss map\/} of $F$.
 Furthermore, we  call $g$ in \eqref{eq:gomega} 
 the {\it secondary Gauss map\/} and $Q=\omega\cdot dg$ 
 the {\it Hopf differential\/} of $F$.
\end{definition}
We remark that the secondary Gauss map $g$ satisfies 
\[
      g=-\frac {dF_{22}}{dF_{21}}=-\frac {dF_{12}}{dF_{11}} \; . 
\]

Let $F\colon{} M \to \SL(2,\C)$ be a null meromorphic map.
Then for $a,b\in \SL(2,\C)$, $\widetilde F =aFb^{-1}$ is also a null
meromorphic map.
Then the associated two Gauss maps $\widetilde G$, $\widetilde  g$, and
the Hopf differential $\widetilde Q$ of $\widetilde F$ are 
\begin{equation}\label{eq:three}
  \widetilde G=a\star G \; , \quad
  \widetilde g=b\star g \; , \quad
  \mbox{and}\quad
  \widetilde  Q=Q \; ,
\end{equation}
where, for any matrix $a=(a_{ij})\in\SL(2,\C)$ and any function $G$,
$a\star G$ is the M\"obius transformation of $G$ by $a$: 
\begin{equation}\label{eq:moebius}
  a\star G = \frac{a_{11}G+a_{12}}{a_{21}G+a_{22}} \; .  
\end{equation}

Let $z$ be a complex coordinate on a neighborhood $U$ of $M$.  
Now we consider the Schwarzian derivatives $S(G)$ and $S(g)$ on $U$ of
$G$ and $g$, where 
\begin{equation}\label{eq:sch}
  S(G)=
  \left[
  \left(\frac{G''}{G'}\right)'-\frac 12 \left(\frac{G''}{G'}\right)^2
  \right]\,dz^2 \qquad
  \left( '=\frac{d}{dz}\right).
\end{equation}
The description of the Schwarzian derivative depends on the choice of
complex coordinate $z$.  
However, any difference of two Schwarzian derivatives, as a meromorphic
$2$-differential, is independent of the choice of complex coordinate.
Note that the Schwarzian derivative is invariant under M\"obius
transformations:
\begin{equation}\label{eq:sch-inv}
  S(G)=S(a\star G) \qquad \bigl(a\in\SL(2,\C)\bigr).
\end{equation}
The following identity can be checked: 
\begin{equation}\label{S-rel}
   S(g)-S(G)=2Q\;.
\end{equation}
Conversely, the following lemma holds:%
\begin{lemma}[\cite{Sm,uy3,kuy}]\label{lem:unique}
 Let $(G,g)$ be a pair of meromorphic functions on $M$ such that
 $S(g)-S(G)$ is not identically zero. 
 Then there exists a unique {\rm (}up to sign{\rm )} null meromorphic
 map $F\colon{}M\to \SL(2,\C)$ such that $G$ and $g$ are the
 hyperbolic Gauss map and the secondary Gauss map of $F$.  
\end{lemma}

Now, for later use, we point out the following elementary fact from
linear algebra: 
\begin{lemma}[\cite{ruy1}]\label{fact}
  A matrix $a\in \SL(2,\C)$ satisfies $a\bar a=\id$
  if and only if it is of the form
  \[
      a=\begin{pmatrix}       
	 p & i \gamma_1 \\
         i\gamma_2 & \bar p 
	\end{pmatrix}
      \qquad \text{with}\quad \gamma_1,\gamma_2\in \R \; , \quad 
      p\bar p+\gamma_1\gamma_2=1 \; , \quad i=\sqrt{-1} \; .  
  \]
\end{lemma}

\subsection*{CMC-1 surfaces in $H^3$}
We identify the Minkowski 4-space $L^4$, which has the canonical
Lorentzian metric $(~\cdot~,~\cdot~)$ of signature $(-,+,+,+)$, 
with the space of $2\times 2$ hermitian matrices $\Herm(2)$.  
More explicitly, $(t,x_1,x_2,x_3)\in L^4$ is identified with the matrix
\[
  \begin{pmatrix}
   t+x_3 &       x_1+ix_2 \\
   x_1-ix_2&       t-x_3
  \end{pmatrix}
	\in \Herm(2) \; .  
\]
The hyperbolic $3$-space can be defined as the upper component 
\[ H^3=\{\xi=(t,x_1,x_2,x_3) \in L^4\,|\,(\xi,\xi)=-1 \,,\;t > 0\,\} \] 
of the hyperboloid in $L^4$ with the
induced metric.  
In $\Herm(2)$ this is represented as 
\[
  H^3=\{\,X\in \Herm(2)\,;\,\det X=1  \, ,\,\, \trace X >0\,\}
     =\{aa^*\,;\,a\in\SL(2,\C)\} \;,
\]
where $a^*={}^t\bar a$.
The complex Lie group $\SL(2,\C)$ acts isometrically on $H^3$ by
$\rho(a)x=a x a^*$, where $a\in \SL(2,\C)$ and $x\in H^3$. 

Let $M$ be a Riemann surface and $F\colon{}M\to \SL(2,\C)$ a null
holomorphic immersion.
Then $f=FF^*\colon{}M\to H^3$ is a conformal \cmcone{} immersion.
Conversely, for any conformal \cmcone{} immersion $f\colon{}M\to H^3$,
there exists a null holomorphic immersion $F\colon{}\widetilde M\to
\SL(2,\C)$ such that $f=FF^*$,
where $\widetilde M$ denotes the universal covering of $M$.
We call $F$ a {\it lift\/} of the conformal \cmcone{} immersion $f$.
Let $\widetilde F$ be another lift of $f$. 
Then, we have the expression $\widetilde F=Fb^{-1}$ for some
matrix $b\in \SU(2)$. 
Let $(g,\omega)$ be the Weierstrass data of the null map of the lift
$F$.
Then the first fundamental form $ds^2$ and the second fundamental form
$I\!I$ are given by (see \cite{uy1}, for example)
\begin{equation}\label{eq:first}
\begin{aligned}
 ds^2&=(1+|g|^2)^2 \omega\cdot\bar \omega\;, \\
 I\!I&=-Q-\overline Q + ds^2\;, 
\end{aligned}
\end{equation}
where $Q=\omega\cdot dg$ is the Hopf differential of $F$.

Let $f=FF^*\colon{}M\to H^3$ be a complete conformal \cmcone{}
immersion whose total Gaussian curvature is finite.
Since the Gaussian curvature $K$ of \cmcone{} surface is non-negative,
finiteness of the total Gaussian curvature is equivalent to
\[ 
   \int_M (-K)\, dA < \infty\;,
\]
where $dA$ is the area element with respect to the first fundamental
form.
Then there is a compact Riemann surface $\overline M$ and a 
finite number of points $\{p_1,$ $\dots,$ $p_n\}$ $\in \overline M$ such
that $M=\overline M\setminus \{p_1,\dots,p_n\}$. 
We call each $p_j$ an {\it end\/} of $f$.
The hyperbolic Gauss map $G$ on $M$ does not necessarily extend
meromorphically on $\overline M$. 
The end $p_j$ is called a {\it regular\/} end if $p_j$ is at most a pole
singularity of $G$, and otherwise is called an {\it irregular} end.  
Namely, an irregular end is an end at which the hyperbolic Gauss map
has an essential singularity.

On the other hand, the Hopf differential $Q$ can always be extended as a
meromorphic $2$-differential on $\overline M$.
We denote by $\ord_p Q$ the order of the first non-vanishing term
of the Laurent expansion of the Hopf differential $Q$ at 
$p \in \overline M$.
(By definition, $\ord_p Q>0$ at zeros of $Q$ and $\ord_p Q<0$ at poles 
of $Q$.)  
The following lemma is well-known 
(cf.~\cite{Bryant}, Lemma 2.3 of \cite{uy1}).
\begin{lemma}\label{lem:irreg}
 An end $p_j$ is regular if and only if $\ord_{p_j} Q \geq -2$.  
\end{lemma}
Now we set $d\sigma_f=(-K) ds^2$. Then $d\sigma^2_f$ is a pseudometric 
of constant curvature $1$ with conical singularities  
(see the appendix of this paper, and also Proposition~4 of \cite{Bryant}).
It follows from \eqref{eq:first} and the Gauss equation that 
\begin{equation} \label{eq:g-metric}
 d\sigma^2_f=\frac {4\,dg\cdot d\bar g}{(1+|g|^2)^2} \; . 
\end{equation}
Hence $d \sigma^2_f$ is the pull-back of the Fubini-Study metric
$d\sigma^2_0$ on $\CP^1$ induced by the secondary Gauss map
$g\colon{}\widetilde{M}^2\to \C\cup \{\infty\} =\CP^1$.
By \eqref{eq:first} and \eqref{eq:g-metric} we have 
\begin{equation}\label{eq:gauss}
   ds^2\cdot d\sigma^2_f= 4\,Q\cdot\overline{Q} \; . 
\end{equation}
In addition to having conical singularities at the ends $p_j$, 
the pseudometric $d\sigma_f^2$ also has a conical singularity at each
umbilic point $q \in M$ of $f$.  The conical order of $d\sigma_f^2$ at 
each point is defined in the appendix of this paper.  
Since $ds^2$ is positive definite at $q$, we have the following: 
\begin{lemma}
 At umbilic points, the conical order of $d\sigma^2_f$ equals the order
 of $Q$. 
\end{lemma}
\section{Reflections of an abstract spherical triangle}
\label{sec:reflection}

In this section, we introduce a method to construct \cmcone{} surface
with irregular ends.
In \cite{ruy1}, examples with regular ends are constructed from the 
holomorphic data $G$ (the hyperbolic Gauss map) and $Q$ (the Hopf
differential).
However, since the hyperbolic Gauss map has essential singularities
at irregular ends, it is hard to find an explicit expression of $G$ in
our case.
Thus, our construction is based on the secondary Gauss map $g$ and the
Hopf differential $Q$. 
Though $g$ is not a well-defined meromorphic function on the surface, 
the pseudometric $d\sigma^2_f$ as in \eqref{eq:g-metric} is a spherical
metric with conical singularities which is well-defined on the surface.
So, we start by constructing a spherical metric $d\sigma_f^2$ with 
conical singularities, using reflections of spherical triangles (see
\cite{uy3}).

\subsection*{Abstract spherical triangles}%
In this section, we consider abstract spherical triangles and their
extensions by reflection.  
First, we shall define abstract triangles.
We set
\[
   \Delta:=\{x+iy\in \C\,;\, x\ge 0,\,y\ge 0,\,x+y\le 1\}\;,
\]
and label each vertex $V_1$, $V_2$, $V_3$ of this closed triangular 
region $\Delta$ as in Figure~\ref{fig:triangle}.
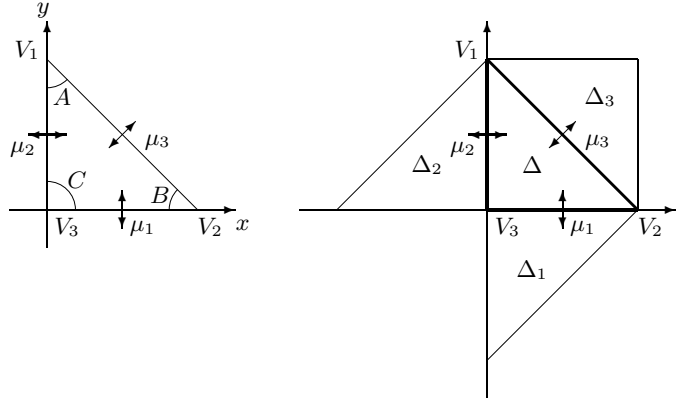
\begin{figure}
\setlength{\unitlength}{5mm}
\small
\begin{center}
\begin{tabular}{cc}
\begin{picture}(11,11)(-5,-5)
 \put(-1,0){\vector(1,0){6}}
 \put(0,-1){\vector(0,1){6}}
 \put(0,4){\line(1,-1){4}}
 \put(0.2,-0.2){\makebox(0,0)[lt]{$V_3$}}
 \put(4,-0.2){\makebox(0,0)[lt]{$V_2$}}
 \put(-0.2,4){\makebox(0,0)[rb]{$V_1$}}
 \put(0,0){\arc{1.5}{-1.5707963}{0}}
 \put(4,0){\arc{1.5}{-3.1415927}{-2.3561945}}
 \put(0,4){\arc{1.5}{-5.4977871}{-4.712389}}
 \put(0.8,0.8){\makebox(0,0)[cc]{$C$}}
 \put(3.0,0.4){\makebox(0,0)[cc]{$B$}}
 \put(0.4,3){\makebox(0,0)[cc]{$A$}}
 \put(5.2,-0.4){\makebox(0,0)[cc]{$x$}}
 \put(-0.1,5.3){\makebox(0,0)[cc]{$y$}}
 \put(2,0){\vector(0,1){0.5}}
 \put(2,0){\vector(0,-1){0.5}}
 \put(0,2){\vector(1,0){0.5}}
 \put(0,2){\vector(-1,0){0.5}}
 \put(2,2){\vector(1,1){0.35355339}}
 \put(2,2){\vector(-1,-1){0.35355339}}
 \put(2.2,-0.3){\makebox(0,0)[lt]{$\mu_1$}}
 \put(-0.3,1.8){\makebox(0,0)[rt]{$\mu_2$}}
 \put(2.6,2){\makebox(0,0)[lt]{$\mu_3$}}
\end{picture}
  &
\begin{picture}(11,11)(-5,-5)
 \put(-5,0){\vector(1,0){10}}
 \put(0,-5){\vector(0,1){10}}
 \Thicklines
  \put(0,4){\line(1,-1){4}}
  \put(0,0){\line(1,0){4}}
  \put(0,0){\line(0,1){4}}
 \thinlines
  \put(4,0){\line(0,1){4}}
  \put(0,4){\line(1,0){4}}
  \put(-4,0){\line(1,1){4}}
  \put(0,-4){\line(1,1){4}}
 \put(0.2,-0.2){\makebox(0,0)[lt]{$V_3$}}
 \put(4,-0.2){\makebox(0,0)[lt]{$V_2$}}
 \put(-0.2,4){\makebox(0,0)[rb]{$V_1$}}
 \put(2,0){\vector(0,1){0.5}}
 \put(2,0){\vector(0,-1){0.5}}
 \put(0,2){\vector(1,0){0.5}}
 \put(0,2){\vector(-1,0){0.5}}
 \put(2,2){\vector(1,1){0.35355339}}
 \put(2,2){\vector(-1,-1){0.35355339}}
 \put(2.2,-0.3){\makebox(0,0)[lt]{$\mu_1$}}
 \put(-0.3,1.8){\makebox(0,0)[rt]{$\mu_2$}}
 \put(2.6,2){\makebox(0,0)[lt]{$\mu_3$}}
 \put(1.2,1.2){\makebox(0,0)[cc]{$\Delta$}}
 \put(-1.6,1.2){\makebox(0,0)[cc]{$\Delta_2$}}
 \put(1.2,-1.6){\makebox(0,0)[cc]{$\Delta_1$}}
 \put(3,3){\makebox(0,0)[cc]{$\Delta_3$}}
\end{picture}
\end{tabular}
\end{center}
    \caption{The triangle $\Delta$ and its reflections.}
    \label{fig:triangle}
\end{figure}            
An {\it abstract spherical triangle\/} is a pair $(\Delta,d\sigma^2)$,
where $d\sigma^2$ is a Riemannian metric defined on $\Delta$ with
constant curvature $1$ such that the three edges forming the 
boundary $\partial \Delta$ are geodesics. 
Let $A$, $B$ and $C$ be the interior angles of $\Delta$ at the
vertices $V_1$, $V_2$, and $V_3$ with respect to the metric $d\sigma^2$
respectively.
The angles $A$, $B$ and $C$ are positive, but may take any positive 
values, including those greater than or equal to $\pi$.  
Moreover, if $A,B,C\not\in \pi \Z$, then we call the abstract spherical
triangle $(\Delta,d\sigma^2)$ {\it non-degenerate}. 
The following fact is known:
\begin{lemma}[\cite{uy7}]
\label{lem:trig}
 Let $(\Delta, d\sigma^2)$ be a non-degenerate abstract spherical
 triangle, then the three angles $A$, $B$, $C$ satisfy the inequality
 \begin{equation}\label{eq:tri}
    \cos^2 A +\cos^2 B + \cos^2 C  +
               2\cos A\cos B \cos C < 1\;.
 \end{equation} 
 Conversely, if a triple of positive real numbers $(A,B,C)$ satisfies
 \eqref{eq:tri}, then there exists a unique non-degenerate abstract
 spherical triangle such that the angles at $V_1$, $V_2$ and $V_3$ are
 $A$, $B$ and $C$ respectively.
\end{lemma}
\begin{proof}[Proof\/\footnotemark{}.]
\footnotetext{
   Recently, an alternative proof and a geometric explanation of 
   this lemma were given in \cite{FH} and \cite{F}.
}
 We take a double (two identical copies) of $(\Delta, d\sigma^2)$ and
 glue them along their corresponding vertices and edges.  
 Then we get a conformal pseudometric on $S^2$ with three conical
 singularities with conical angles $2A$, $2B$, $2C$. 
 Consequently, the conical orders are
\[
 \frac{A-\pi}\pi \; ,\quad \frac{B-\pi}\pi \; ,\quad \frac{C-\pi}\pi
\]
 respectively.  
 If $(\Delta, d\sigma^2)$ is non-degenerate, then $A,B,C\not \in \pi \Z$. 
 By Corollary~3.2 of \cite{uy7}, the metric on $S^2$ is irreducible. 
 Then \eqref{eq:tri} follows from (2.19) of \cite{uy7}.

 Conversely, suppose that \eqref{eq:tri} holds.
 By Theorem 2.4 of \cite{uy7}, there exists a unique conformal
 pseudometric on $S^2$ with three conical singularities with conical
 angles $2A, 2B, 2C\not\in 2\pi \Z$.  
 The uniqueness of this pseudometric implies that it has a symmetry and 
 can be considered as a gluing of two identical non-degenerate spherical
 triangles.  
\end{proof}
The above lemma implies that a non-degenerate abstract spherical 
triangle is uni\-quely determined by its angles $A,B,C$.  
So we denote it by 
\[
    \Trig(A,B,C):=(\Delta,d\sigma^2).
\]
Now we fix a non-degenerate abstract spherical triangle
$\Trig(A,B,C)$. 
Since $\Delta$ is simply connected, there exists a meromorphic function
$g\colon{}\Delta\to\C\cup\{\infty\}$ such that the pull-back of the
Fubini-study metric on $\C\cup\{\infty\}=\CP^1$ by $g$ is
$d\sigma^2$.
However, such a choice of the developing map has an ambiguity up to an
$\SU(2)$-matrix action $g\mapsto a\star g$ ($a\in\SU(2)$).
We shall now remove this ambiguity, using a normalization:
There exists a unique (up to sign) developing map 
\[
   g=g_{A,B,C}:\Delta\longrightarrow \C\cup\{\infty\}
\]
of $d\sigma^2$ satisfying (see Figure \ref{fig:triangle})
\begin{equation}\label{eq:normalize}
   e^{-i C}g(V_1)\in \R\cup\{\infty\}\;,\quad
   g(V_2)\in \R\cup\{\infty\}\quad
   \text{and}\quad
   g(V_3)=0\;.
\end{equation}
We call the developing map $g=g_{A,B,C}$ the 
{\it normalized developing map\/} of the triangle $\Trig(A,B,C)$.  

Let $\mu_j$ ($j=1,2,3$) be the reflections of the triangle
$\Trig(A,B,C)$ across the three edges, as in Figure~\ref{fig:triangle}.
Let $\Delta_j$ be the closed domain obtained by reflecting $\Delta$ with
respect to $\mu_j$ (see Figure~\ref{fig:triangle}).
Then each reflection $\mu_j$ can be regarded as an involution on the
domain $\Delta\cup\Delta_j$.  
\begin{lemma}[Monodromy principle]\label{lem:p-I}
 Let $\Trig(A,B,C)$ be a non-degenerate abstract spherical triangle
 and $g_{A,B,C}:\Delta \to \C\cup \{\infty\}$ $(j=1,2,3)$  be the
 normalized developing map of $\Trig(A,B,C)$.
 Then the following identities hold\/{\rm :} 
 \begin{align*}
  \overline{ g_{A,B,C}\circ \mu_1}&=g_{A,B,C}\;,\\
  \overline{ g_{A,B,C}\circ \mu_2}&=e^{-2iC}g_{A,B,C}\;,\\
  \overline{ g_{A,B,C}\circ \mu_3}&=
  \begin{pmatrix}
                q & i \delta \\
               i\delta & \bar q
  \end{pmatrix} \star g_{A,B,C}
      \qquad (\delta\in \R, \quad q\bar q+\delta^2=1)\;,
 \end{align*}  
 where
 \begin{equation}\label{eq:q-value}
     q =  \frac{i}{\sin C}(\cos A + e^{iC}\cos B)\;.
 \end{equation}
\end{lemma}
\begin{proof}
 To simplify the notation, we set $g=g_{A,B,C}$.
 Let $\gamma_1$, $\gamma_2$ and $\gamma_3$ be the 
 three edges of $\Delta$
 which are stabilized by the reflections $\mu_1$, $\mu_2$ and $\mu_3$,
 respectively. 
 Since the edge $\gamma_1$ is a geodesic, \eqref{eq:normalize} implies
 that $g(\gamma_1)$ lies on the real axis.
 Then by the reflection principle, $\overline{g\circ\mu_1}=g$ holds.
 Similarly, by \eqref{eq:normalize}, $e^{-iC}g(\gamma_2)$ is real.
 Hence 
 \[
    \overline{e^{-iC}g\circ\mu_2}=e^{-iC}g
 \]
 holds.
 Then we have the second assertion.  

 There exists a rotation $a$ centered at $g(V_2)$ of the unit 2-sphere 
 $S^2(=\C\cup\{\infty\})$ such that the image of $g(V_1)$ is real.  
 Such an isometry $a$ of $S^2$ can be represented as a matrix 
 $a\in \SU(2)$. 
 Then we have $a\star g(\gamma_3)$ lies on the real axis.
 Hence by the reflection principle,
 $\overline{a\star g\circ\mu_3}  = a \star g$ holds, and then we have
 \[
     \overline{g\circ\mu_3} =
        \left(\bar a^{-1}a\right)\star g\;.
 \]
 In particular, we have
 \begin{equation}\label{eq:easy}
   \overline{ g\circ \mu_1}=g,\quad
   \overline{ g\circ \mu_2}=e^{-2iC}g,\quad
   \overline{ g\circ \mu_3}=(\bar a^{-1}a)\star g\;.
 \end{equation}
 Now we glue $\Trig(A,B,C)$ to a double of $\Trig(A,B,C)$ along 
 corresponding edges and vertices, giving us 
 a constant curvature one metric of three conical singularities
 on $S^2$ with conical angles $2A$, $2B$, and $2C$ just as 
 in the proof of Lemma \ref{lem:trig}.  
 The open domains $\Delta_1$, $\Delta_2$ and $\Delta_3$ can be regarded
 as the interior of the second triangle in the double 
 of $\Trig(A,B,C)$.  
 The monodromy of reflections for such metrics on $S^2$ are determined
 in \cite{uy7}.
 Then, as shown at the bottom of page~82 of \cite{uy7},  we have
 \[
   \bar a^{-1}a=
   \begin{pmatrix}
    q & i \delta \\
    i\delta & \bar q
   \end{pmatrix},
   \qquad (\delta\in \R,\, q\bar q+\delta^2=1)
 \]
 with $q$ as in \eqref{eq:q-value}.
\end{proof}
\subsection*{Closed Riemann surfaces generated by three reflections}
Let $\overline M$ be 
a closed Riemann surface and $D(\subset \overline{M})$
a simply connected closed domain bounded by three real analytic curves
$\gamma_1$, $\gamma_2$ and $\gamma_3$.
We label the vertices $V_1$, $V_2$, $V_3$ of this closed 
triangular region 
$D$ such that $\gamma_1$, $\gamma_2$, and $\gamma_3$ correspond to the
three edges $V_2V_3$, $V_3V_1$ and $V_1V_2$, respectively.
The Riemann surface $\overline M$ is generated by $D$ if there are three
anti-holomorphic reflections $\mu_1$, $\mu_2$, $\mu_3$ of $\overline M$
stabilizing the three edges $\gamma_1$, $\gamma_2$, and $\gamma_3$ of 
$D$ such that any point of 
$\overline M$ can be contained in the image of $D$ by a suitable finite
composition of these three reflections (see Figure~\ref{fig:domain}).  
\begin{figure}
\begin{center}
\includegraphics[width=4cm]{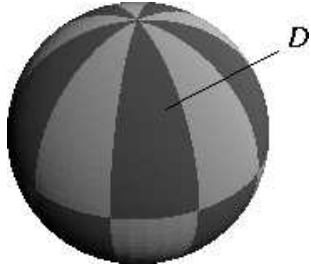}
\end{center}
\caption{A Riemann surface generated by reflections:
         Here the Riemann surface is the sphere $S^2$, and is obtained 
         from the triangle $D$ by reflections.}
\label{fig:domain}
\end{figure}
In this case, $D$ is called a {\it fundamental domain\/} of 
$\overline M$.
A meromorphic $2$-differential $Q$ on $\overline M$ is said to be
symmetric with respect to $D$ if
\[
     \overline{Q\circ \mu_j}=Q\qquad (j=1,2,3)
\]
holds. 

We let $\metone(\overline M)$ denote the set of conformal pseudometrics
with conical singularities on $\overline M$.  
A metric $d\sigma^2 \in \metone(\overline M)$ is called 
{\it symmetric\/} with respect to the fundamental domain 
$D$ if it is invariant under the three 
reflections $\mu_1$, $\mu_2$ and $\mu_3$.  
Moreover, if the restriction $(D, d\sigma^2|_{D})$ is isometric to 
$\Trig(A,B,C)$, we denote the metric by 
\[
   d\sigma^2=d\sigma^2_{A,B,C}\;.
\]

A meromorphic function $g$ on $\overline M$ is called 
{\it $\SU(2)$-symmetric\/} with respect to $D$ if 
\begin{equation}\label{eq:dsigmag}
   d\sigma^2_g :=\frac{4\,dg\cdot d\bar g}{(1+|g|^2)^2}
\end{equation}
is symmetric with respect to $D$.  

The following is the main theorem in this paper:
\begin{theorem}\label{thm:gen}
 Let $\overline M$ be a closed Riemann surface which is generated 
 by a triangular fundamental domain $D\subset\overline M$ by 
 reflections, and label the vertices of $D$ as $V_1$, $V_2$ and $V_3$.
 Let $g$ be an $\SU(2)$-symmetric meromorphic function on $\overline M$
 with respect to $D$, and let $Q$ be a symmetric meromorphic 
 $2$-differential on $\overline M$ with respect to $D$.  
 Suppose that{\rm :}
 \begin{enumerate}
\item\label{item:gen:angle}
      There exist $A,B_0 \in \R^+ \setminus \pi \Z$ such that
      $d\sigma^2_g=d\sigma^2_{A,B_0,\pi/2}$, 
      with $d\sigma^2_g$ as in \eqref{eq:dsigmag}.
 \item\label{item:gen:ord} 
      $Q$ is holomorphic on $D\setminus \{V_2\}$ and has a pole at
      $V_2$ with $\ord_{V_2}Q\leq -3$.
 \item\label{item:gen:branch} 
      The branch point set of $g$ outside the poles of $Q$
      equals to the zero set of $Q$, and at each zero of $Q$,
      the order of $Q$ is equal to the conical order of $d\sigma^2_g$.
 \setcounter{tempcounteri}{\value{enumi}}
 \end{enumerate}
 Let $p_1,\dots p_n \in \overline{M}$ be the set of poles of $Q$.
 Then, for some $\varepsilon>0$, there exist a 
 smooth function $B(t)\colon{}(-\varepsilon,\varepsilon)\to\R$
 satisfying $B(0)=B_0$ 
 and a $1$-parameter family of conformal \cmcone{} immersions 
 $f_t:\overline M\setminus\{p_1,\dots,p_n\} \to 
 H^3$ for $t\in (-\varepsilon,\varepsilon)$ with the following
 properties{\rm :}
 \begin{enumeraterom}
  \item\label{item:gen:concl-1} 
       The Hopf differential of $f_t$ is $tQ$ and 
       $d\sigma^2_{f_t}=d\sigma_{A,B(t),\pi/2}$.
  \item\label{item:gen:concl-3} 
       $f_t$ has irregular ends at $\{p_1,\dots,p_n\}$.
  \item\label{item:gen-concl-4}
       $f_t$ is symmetric with respect to $D$.
       That is, the image of $f_t$ is generated by the reflections
       across the edges of $f_t(D)$.
 \setcounter{tempcounterii}{\value{enumi}}
 \end{enumeraterom}
\end{theorem}
\begin{remark}
 The construction method for proving Theorem~\ref{thm:gen} will 
 still work if $\ord_{V_2}Q=-2$. %
 The stronger assumption $\ord_{V_2}Q\leq -3$ in \ref{item:gen:ord}
 is required only to make the ends irregular
(see Lemma~\ref{lem:irreg}).
\end{remark}
The next theorem gives conditions which imply the surfaces $f_t$ in
Theorem~\ref{thm:gen} are irreducible.
\begin{theorem}
\label{thm:irred}
 Under the assumptions in Theorem~{\rm \ref{thm:gen}}, if
 $A\not\equiv \pi/2 \pmod{\pi}$ and 
\begin{equation}\label{eq:irred-cond}
         \oint_{\tau}
           \begin{pmatrix}
		  g & - g^2 \\
                  1 & -g\hphantom{^2}
	   \end{pmatrix}\frac{Q}{dg}\neq 0\;
\end{equation}
  for a local loop $\tau$ about $V_2$,  then 
  $f_t$ is irreducible for $t$ sufficiently close to zero.
\end{theorem}
The proofs of these theorems are given in Section~\ref{sec:proof}.
\section{CMC-1 surfaces with dihedral and Platonic symmetries}
\label{sec:platonic}
In this section, we construct examples of finite total curvature 
\cmcone{} surfaces with irregular ends and either dihedral or Platonic
symmetries.
Examples with dihedral symmetries and the simplest example with
tetrahedral symmetry (the case $m=1$ in \eqref{eq:tetraQ} and
\eqref{eq:tetrag}) are irreducible.  
(Though we expect all the other examples to be irreducible, we have not
checked them yet.)

\subsection*{CMC-1 surfaces with dihedral symmetries}
Let $n\geq 3$ be an integer and 
\[
     M_n:=\C\cup\{\infty\}\setminus\{1,\zeta,\dots,\zeta^{n-1}\}
     \qquad \left(\zeta=\exp\frac{2\pi}{n} i\right)\;.
\]
The Jorge-Meeks' $n$-noid is the complete minimal immersion 
$f_{n,0}\colon{}M_n\to\R^3$ given by the Weierstrass representation
as
\begin{multline*}
     f_{n,0}:= \Re\int
               \bigl((1-g_{n,0}^2), i(1+g_{n,0}^2),2 g_{n,0}\bigr)
               \frac{Q_{n,0}}{dg_{n,0}}\;,\\
     \text{where}\qquad
     g_{n,0} = z^{n-1},\quad
     Q_{n,0} = \frac{z^{n-2}}{(z^n-1)^2}\,dz^2\;.
\end{multline*}
The $\Z_2$ extension $D_n \times \Z_2$ of the dihedral group 
$D_n$ acts isometrically on the image of $f_{n,0}$.  

There exists a one-parameter family of corresponding 
\cmcone{} immersions of $M_n$ to $H^3$ such that  the
hyperbolic Gauss map is $g_{n,0}$ and the Hopf differential is
proportional to $Q_{n,0}$ (see \cite{uy3,ruy1}, and
Figure~\ref{fig:trinoid} for the $n=3$ case).
Since $\ord_{\zeta^j}Q_{n,0}=-2$ $(j=0,\dots,n-1)$, the ends of 
these corresponding surfaces are regular.
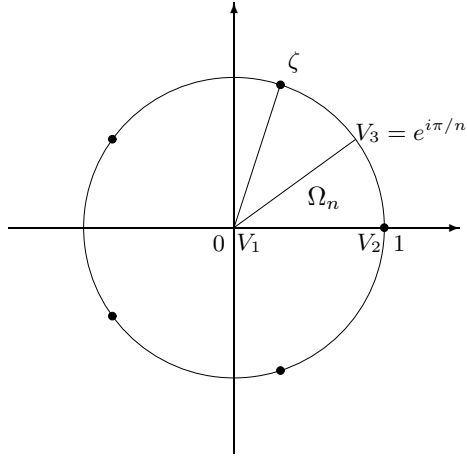
\begin{figure}
 \setlength{\unitlength}{2cm}
 \small
 \begin{center}
  \begin{picture}(3,3)(-1.5,-1.5)
   \put(-1.5,0){\vector(1,0){3}}
   \put(0,-1.5){\vector(0,1){3}}
   \put(0,0){\circle{2}}
   \put(1,0){\circle*{0.05}}
   \put(0.30901699,0.95105652){\circle*{0.05}}
   \put(0.30901699,-0.95105652){\circle*{0.05}}
   \put(-0.80901699,0.58778525){\circle*{0.05}}
   \put(-0.80901699,-0.58778525){\circle*{0.05}}
   \drawline(0,0)(1.2,0)
   \drawline(0,0)(0.30901699,0.95105652)
   \drawline(0,0)(0.80901699,0.58778525)
   \put(1.1,-0.1){\makebox(0,0){$1$}}
   \put(0.4,1.1){\makebox(0,0){$\zeta$}}
   \put(-0.1,-0.1){\makebox(0,0){$0$}}
   \put(0.6,0.2){\makebox(0,0){\normalsize $\Omega_n$}}
   \put(1.18,0.65){\makebox(0,0){$V_3=e^{i \pi/n}$}}
   \put(0.1,-0.1){\makebox(0,0){$V_1$}}
   \put(0.9,-0.1){\makebox(0,0){$V_2$}}
  \end{picture} 
 \end{center}
 \caption{The fundamental domain of surfaces with dihedral symmetry}
 \label{fig:fund-jm}
\end{figure}

However, as we wish to produce \cmcone{} surfaces in $H^3$ whose 
ends are {\em not} regular, 
we now modify our choices for $Q$ and $g$ to accomplish this:  
Let $m\geq 1$ be an integer and set
\[
     Q_{n,m}:= \frac{z^{n(m+1)-2}}{(z^n-1)^{2(m+1)}}\,dz^2,\qquad
     g_{n,m}:= z^{n(m+1)-1}\;.
\]
Then $g=g_{n,m}$ and $Q=Q_{n,m}$ are a meromorphic function and 
a meromorphic $2$-differential on $\overline M=\C\cup\{\infty\}$
respectively,
which are symmetric with respect to the fundamental domain $\Omega_n$ as
in Figure~\ref{fig:fund-jm}.
Moreover, $(g,Q)$ satisfies the assumptions 
\ref{item:gen:angle}--\ref{item:gen:branch} of 
Theorem~\ref{thm:gen} for $A=\pi(m+1)-\pi/n$
and $B_0=\pi/2$.
Then, for each $n$ and $m$, there exists a one-parameter family of 
\cmcone{} immersions $f_{n,m,t}\colon{}M_n\to H^3$ ($0<|t|<\varepsilon$)
with symmetry group 
$D_n \times \Z_2$ whose Hopf differential is $tQ_{n,m}$
(see Figure~\ref{fig:trinoid}).  The total Gaussian curvature of these 
surfaces will be approximately $4\pi (n(m+1)-1)$.  
\begin{figure}
\small
\begin{center}
\begin{tabular}{c@{\hspace{1cm}}c}
  \includegraphics[width=3cm]{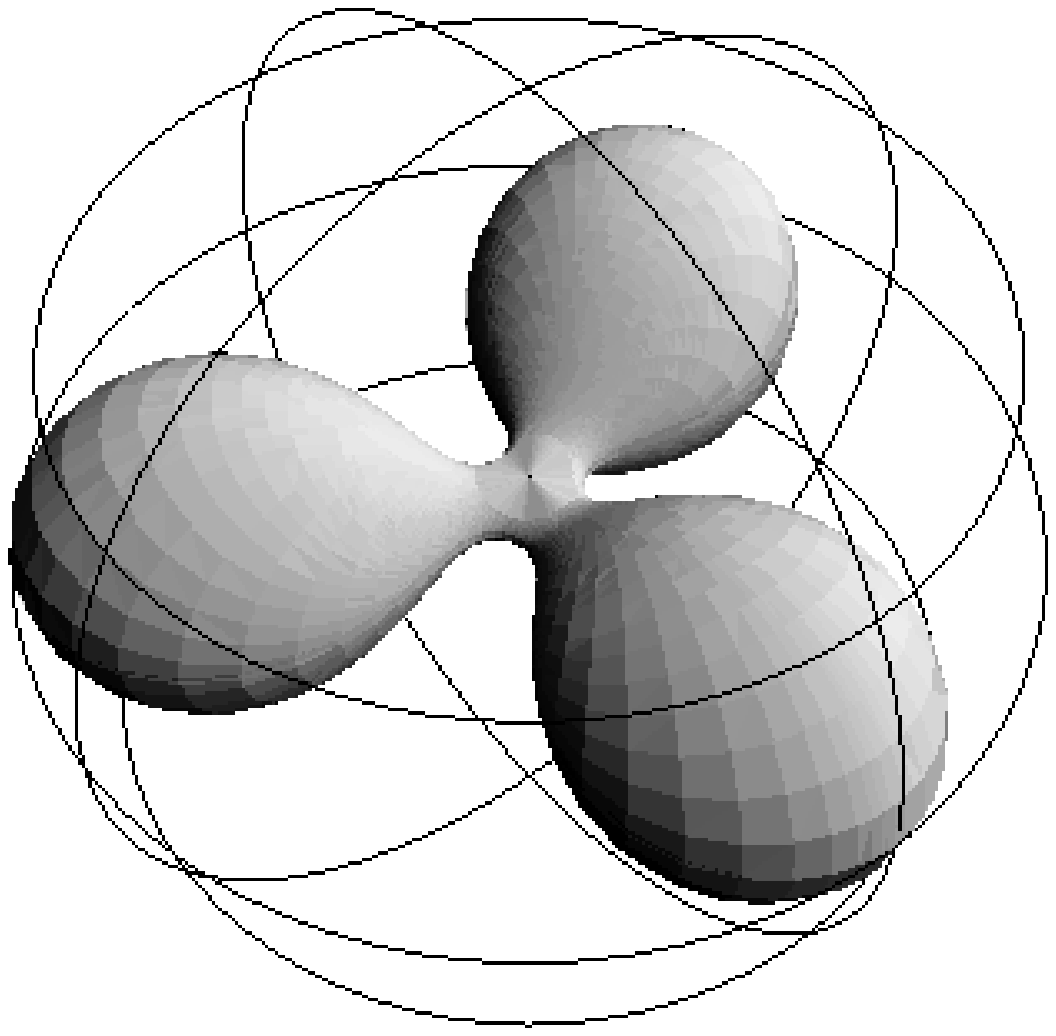} &
  \includegraphics[width=3cm]{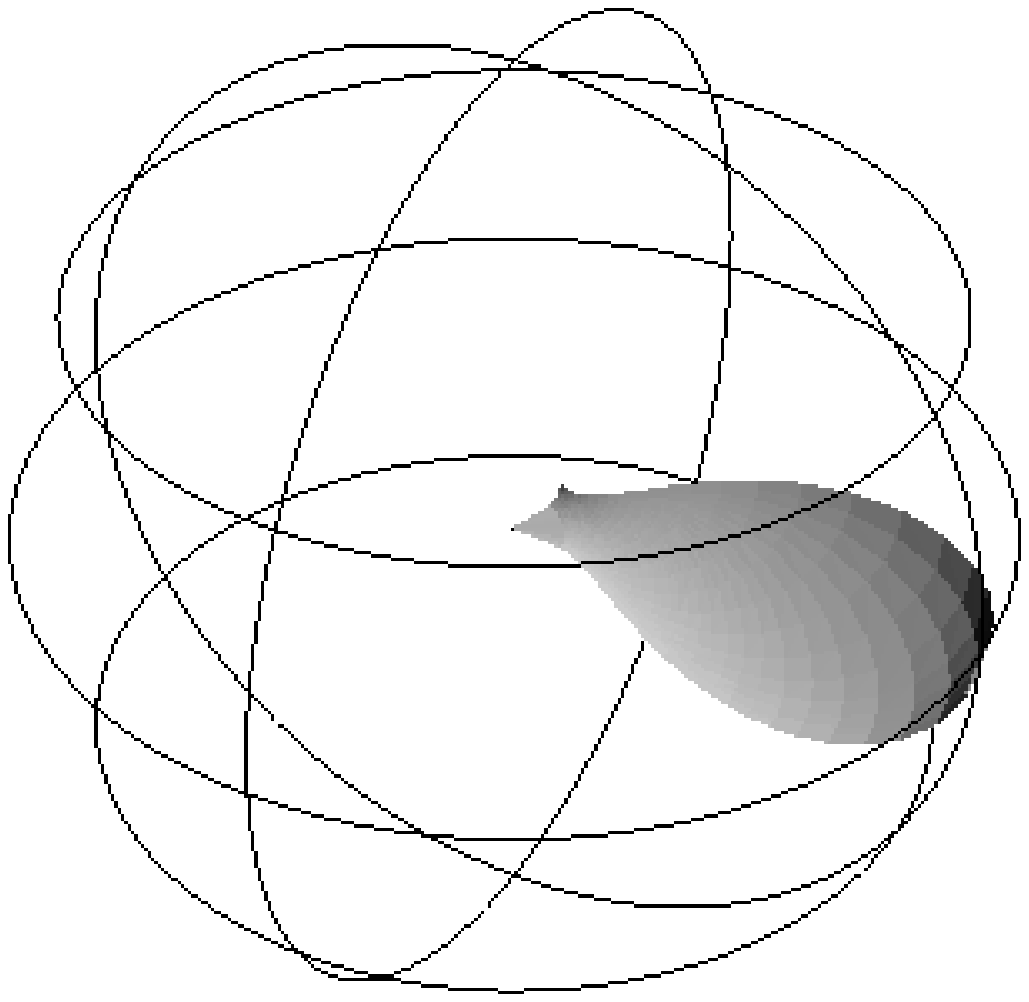} \\
  (a) $f_{3,0,t}$ &
  (b) Fundamental piece of $f_{3,0,t}$ \\[2ex]
  \includegraphics[width=3cm]{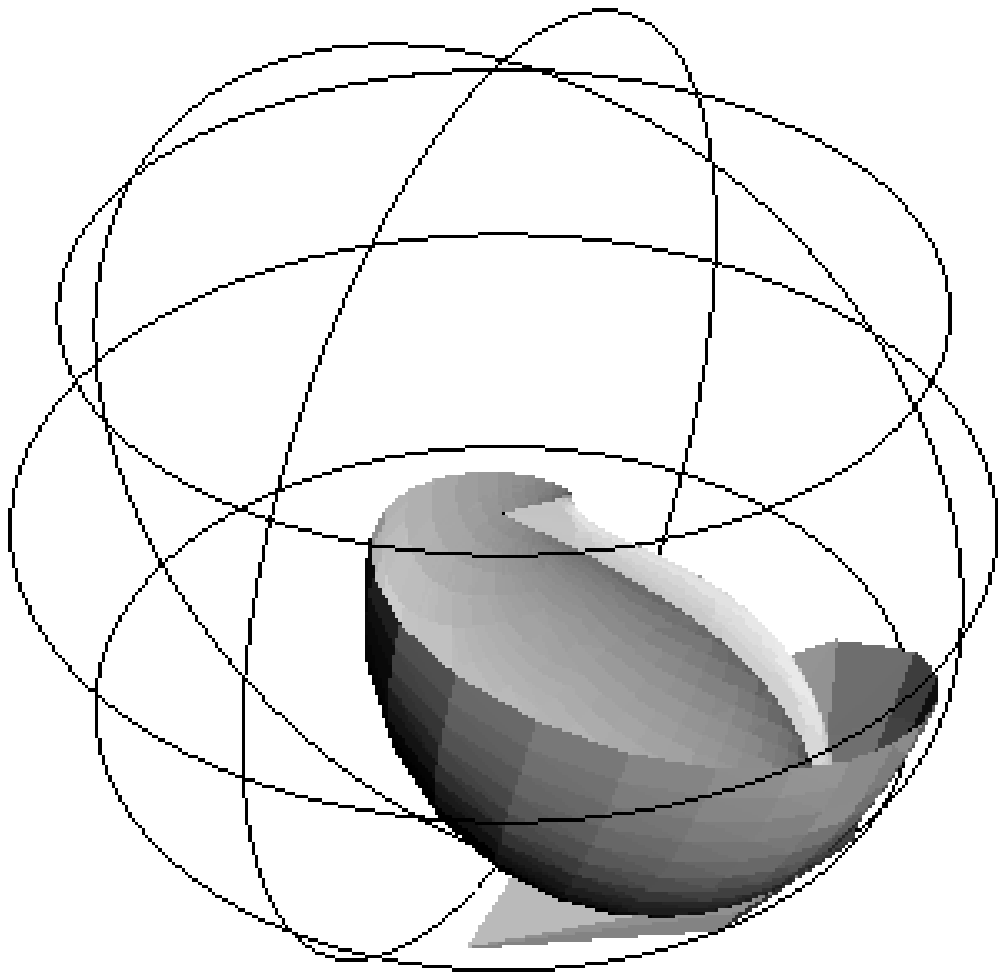} &
  \includegraphics[width=3cm]{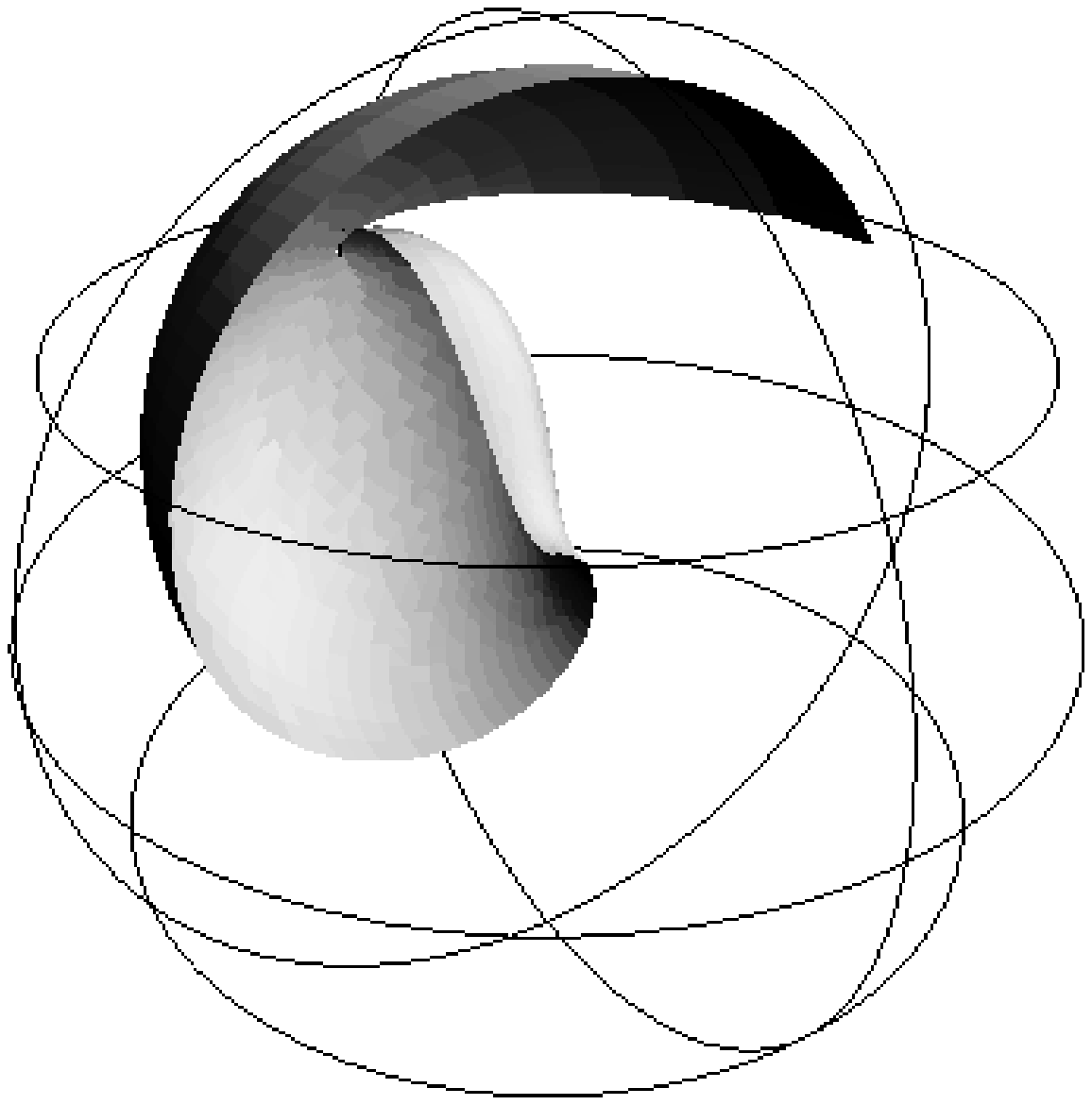} \\
  (c) Fundamental piece of $f_{3,1,t}$ &
  (d) Enneper's cousin
\end{tabular}
\end{center}
\caption{Surfaces with dihedral symmetry:
         Figure (a) shows a \cmcone{} surface with 
         dihedral symmetry and 
         three regular ends (the surface corresponding to 
         a Jorge-Meeks surface) in the Poincar\'e model of $H^3$, 
         and figure (b) is the fundamental region of the surface in 
         figure (a).  
         Figure (c) shows the fundamental piece of $f_{3,1,t}$,
         a surface of dihedral symmetry with three irregular ends.  
         The central part of $f_{3,1,t}$ is similar to that of the 
         hyperbolic correspondence of a Jorge-Meeks surface.  
         The end of $f_{3,1,t}$ seen here is similar to the end of an 
         Enneper cousin \cite{Bryant}, which is shown in figure (d).}
\label{fig:trinoid}
\end{figure}

Let $\tau$ be a loop surrounding the end $V_2=1$.  Since 
\[
    \frac{Q_{n,m}}{dg_{n,m}}
         =\left(\frac{1}{n(m+1)-1}\right)\frac{dz}{(z^n-1)^{2(m+1)}}\;,
\]
we have
\[
    \oint_{\tau}\frac{Q_{n,m}}{dg_{n,m}}=
       \left(\frac{2\pi i}{n(m+1)-1}\right)
       \Res_{z=1} \frac{1}{(z^n-1)^{2(m+1)}} \neq 0\;.
\]
Thus by Theorem~\ref{thm:irred}, the surfaces are irreducible 
for sufficient small $t$.
\subsection*{CMC-1 surfaces with tetrahedral symmetries}
It is well-known that there exists a minimal immersion
\[
   f_{0}\colon{}M
     :=\C\cup\{\infty\}\setminus\{p_1,\dots,p_4\}
     \longrightarrow
     \R^3
\]
with $4$ catenoid ends and tetrahedral symmetry \cite{Kat,Xu,BR,uy3}
and corresponding \cmcone{} surfaces in $H^3$ \cite{uy3,ruy1} with
regular ends.

We denote by $Q_{0}$ and $g_{0}$ the Hopf differential and the
Gauss map of $f_{0}$ respectively.
Since each end is asymptotic to a catenoid, $\ord_{p_j}Q_{0}=-2$
($j=1,\dots,4$).
Then there exists $4$ umbilic points (zeros of $Q_{0}$)
$q_1,\dots,q_4$ such that $\ord_{q_j}Q_{0}=1$ ($j=1,\dots,4$).
The Gauss map $g_{0}$ is a meromorphic function on
$\C\cup\{\infty\}$ whose branch points are $\{q_1,\dots,q_4\}$ 
each with branch order $1$.
\begin{figure}
 \setlength{\unitlength}{7mm}
 \small
 \begin{center}
  \begin{picture}(7,7)
   \Thicklines
   \drawline(4,6)(4.5,1)(1,2)(4,6)
   \drawline(4,6)(6,3)(4.5,1)
   \thinlines
   \dottedline{0.1}(1,2)(6,3)
   \drawline(4,6)(2.75,1.5)
   \drawline(4.5,1)(2.5,4)
   \drawline(1,2)(4.25,3.5)
   \put(4,6.4){\makebox(0,0){$p_1$}}
   \put(4.5,0.6){\makebox(0,0)[cc]{$p_2$}}
   \put(0.6,2){\makebox(0,0)[cc]{$p_3$}}
   \put(6.4,3){\makebox(0,0)[cc]{$p_4$}}
   \put(2.7,3.15){\makebox(0,0)[cc]{$q_1$}}
   \put(3.8,4){\makebox(0,0)[cc]{\normalsize$D$}}
  \end{picture} 
 \end{center}
 \caption{The fundamental domain of surfaces with tetrahedral symmetry}
 \label{fig:fund-tet}
\end{figure}
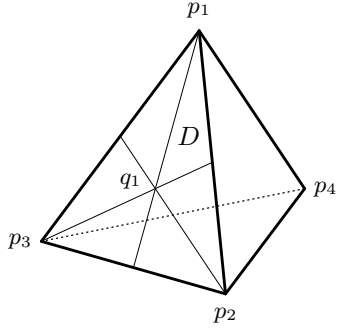
Moreover, $M$ is obtained by reflections across the edges of the 
fundamental domain $D$, which is a triangle with vertices 
$V_1=q_1$, $V_2=p_1$, $V_3$.  (See Figure~\ref{fig:fund-tet}. 
See also the construction in \cite{uy3}.) 
The Hopf differential $Q_{0}$ and the Gauss map $g_{0}$ are symmetric
with respect to the fundamental domain $D$.

Consider the Schwarzian derivative $S(g_{0})$ of $g_{0}$, as in 
\eqref{eq:sch}, 
where $z$ is the usual complex coordinate of $\C\cup\{\infty\}$.
Then $S(g_0)$ is a meromorphic $2$-differential on $\C\cup\{\infty\}$.
Moreover, since $g_{0}$ is symmetric, $S(g_0)$ is invariant under
reflections about the edges of $D$.

The branch points of $g_0$ are poles of $S(g_0)$, 
and each pole of $S(g_{0})$ has order $-2$.  
So $S(g_0)$ has $4$ poles of order $2$ at the $q_j$ ($j=1,\dots,4$) and
is holomorphic on $\C\cup\{\infty\}\setminus\{q_1,\dots,q_4\}$.
Since the total order of a meromorphic $2$-differential on
$\C\cup\{\infty\}$ is $-4$, $S(g_0)$ has $4$ zeros (counting
multiplicity).
If there exists a zero of $S(g_0)$ on the interior of $D$, $S(g_0)$ has
at least $24$ zeros because $\C\cup\{\infty\}$ consists of $24$ copies
of the fundamental region $D$, which is impossible.
Similarly, if a zero of $S(g_0)$ lies on the interior of a edge of $D$,
$S(g_0)$ has at least $12$ zeros, which is also impossible.
If the vertex $V_3$ of $D$ is a zero of $S(g_0)$, there exist $6$ zeros,
which is again impossible.  
Since $V_1=q_1$ is a pole of $S(g_0)$, the set of zeros of $S(g_0)$ must
be $\{p_1,\dots,p_4\}$, and $\ord_{p_j}S(g_0)=1$ for $j=1,\dots 4$.  

For an integer $m\geq 1$, we define a meromorphic $2$-differential
$Q_{m}$ as
\begin{equation}\label{eq:tetraQ}
    Q_{m}:=\frac{Q_{0}{}^{m+1}}{S(g_0)^m}\;,
\end{equation}
where  $Q_{0}{}^{m+1}$ (resp.~$S(g_0)^m$) is the symmetric product of
$m+1$ copies of $Q_{0}$ (resp.~$m$-copies of $S(g_0)$).
Since the poles of $Q_{0}$ and the zeros of $S(g_0)$ are
$\{p_1,\dots,p_4\}$,
$Q_{m}$ has the poles $\{p_1,\dots,p_4\}$ and is holomorphic on $M$.
More precisely, 
\[
   \ord_{p_j}Q_{m}= -3m-2\qquad\text{and}\qquad
   \ord_{q_j}Q_{m}= 3m+1\qquad (j=1,\dots,4)
\]
hold.
Since $Q_{0}$ and $S(g_0)$ are invariant under the reflections,
so is $Q_{m}$.

Consider an abstract spherical triangle $\Trig(A,B_0,C)$ with
\begin{equation}\label{eq:tetrag}
   A = \pi m + \frac{2}{3}\pi\;,\quad
   B_0 = \frac{\pi}{3}\;,\quad
   C = \frac{\pi}{2}\;,
\end{equation}
and identify it with the fundamental domain $D$.
Then we have a pseudometric
$d\sigma^2_{A,B_0,C}\in\metone(\C\cup\{\infty\})$.
Since neighborhoods of $V_1$, $V_2$ and $V_3$ are generated by 
$6$, $6$ and $4$ copies of the fundamental domain $D$, respectively, 
$d\sigma^2_{A,B_0,C}$ is a pseudometric whose conical orders are all
integers.
Hence by Proposition~\ref{prop:A} in Appendix, the developing map 
$g_{m}$ of $d\sigma^2_{A,B_0,C}$ is meromorphic on $\C\cup\{\infty\}$.

Then one can easily check that $(g_{m},Q_{m})$ satisfies 
the assumptions \ref{item:gen:angle}--\ref{item:gen:branch} of 
Theorem \ref{thm:gen}.  
Hence for each $m$, there exists a one-parameter family of \cmcone{}
immersions $\{f_{m,t}\,;\,0<|t|<\varepsilon\}$ of $M$ into $H^3$ with
irregular ends and tetrahedral symmetry.

Finally, we check irreducibility for $m=1$.
We may set  
\[
   M = \C\cup\{\infty\}\setminus\{1,\zeta,\zeta^2,\infty\}\;,
   \qquad\text{where}\quad \zeta =\exp\frac{2}{3}\pi i\;,
\]
and 
\[
    g_0 = \frac{1}{3\sqrt{2}}\left(z-\frac{4}{z^3}\right),\qquad
    Q_0 = \frac{z(z^3+8)}{(z^3-1)^2}\,dz^2
\]
(see page 221 of \cite{uy3}).
Hence the umbilic points are  $\{0,-2,-2\zeta,-2\zeta^2\}$.
By direct calculation, we have
\[
    Q_1 = \frac{Q_0{^2}}{S(g_0)}=
    \frac{1}{96}\frac{z^4(z^3+8)^4}{(z^3-1)^5}\,dz^2\;.
\]
On the other hand, $g_{1}$ is a meromorphic function which branches at 
the umbilic points with  branch order $4$.
Then we have $\deg g_1=9$ by the Riemann-Hurwitz formula.
Choose a rotation $a\in\SU(2)$ such that $a\star g_1(q_1)=\infty$,
where we set $q_1=0$.
Then $q_j$ $(j=2,3,4)$ are not poles of $g:=a\star g_1$,
because the multiplicity of $g$ at $q_j$ is $5$ for each $j$ 
and $\deg g=9$.
Moreover, $d\sigma^2_{g_1}=d\sigma^2_g$ has a conical singularity at $0$
with conical order $4$.
Hence, by symmetricity, we have $g_1(\zeta z)=\zeta g_1(z)$, and we may write
\[
    dg = \beta \frac{(z^3+8)^4}{z^6(z^3-a^3)^2}\,dz
\]
for some nonzero constants $a$ and $\beta$.
Such a function $g$ exists if and only if all residues of the right-hand
side vanish, which is equivalent to $a^3=16$.
Then one can check irreducibility by direct calculation and 
Theorem~\ref{thm:irred}.
\subsection*{CMC-1 surfaces with Platonic symmetries}
There are genus zero minimal surfaces in $\R^3$ with catenoidal ends and
the symmetry of any Platonic solid (\cite{Kat,Xu,BR,uy3}).
By similar arguments to the tetrahedral case above, one can obtain 
\cmcone{}
immersions with irregular ends and any Platonic symmetry.
Table~\ref{tab:platonic} shows the data for such surfaces.
\begin{table}
\footnotesize
\begin{center}
\newcommand{\vr}{\rule[-1.5ex]{0mm}{2.8ex}}
 \begin{tabular}{|c||c|c|c|c|c|c|c|}
 \hline
  Symmetry &
  {$\# \{p_j\}$}& 
  {$\# \{q_j\}$}& 
  {$\ord_{q_j}Q_0$} &
  {$\ord_{p_j}Q_m$} &
  {$\ord_{q_j}Q_m$} &
  {$A$} &
  {$B_0$} \\
 \hline
 \hline
  Dihedral &
  $n$ & $2$ & $n-2$ & $2(m+1)$ & $n(m+1)-2$ & 
  $\pi\left(m+1-\frac{1}{n}\right)$ & $\frac{\pi}{2}$\vr\\
 \hline
  Tetrahedral &
  $4$ & $4$ & $1$ & $-3m-2$ & $3m+1$ & $\pi\left(m+\frac{2}{3}\right)$ &
  $\frac{\pi}{3}$\vr\\
 \hline
  Octahedral &
  $8$ & $6$ & $2$ & $-3m-2$ & $4m+2$ & $\pi\left(m+\frac{4}{3}\right)$ &
  $\frac{\pi}{3}$\vr\\
 \hline
  Octahedral &
  $6$ & $8$ & $1$ & $-4m-2$ & $3m+1$ & $\pi\left(m+\frac{2}{3}\right)$ &
  $\frac{\pi}{4}$\vr\\
 \hline
  Icosahedral &
  $20$ & $12$ & $3$ & $-3m-2$ & $5m+3$ & $\pi\left(m+\frac{4}{5}\right)$
  & $\frac{\pi}{3}$\vr\\
 \hline
  Icosahedral &
  $12$ & $20$ & $1$ & $-5m-2$ & $3m+1$ & $\pi\left(m+\frac{2}{3}\right)$
  & $\frac{\pi}{5}$\vr\\ 
 \hline
 \end{tabular}
\end{center}
\caption{Data for \cmcone{} surfaces with Platonic symmetries}
\label{tab:platonic}
\end{table}
\section{Proof of the main theorem}
\label{sec:proof}
\subsubsection*{Proof of Theorem~\ref{thm:gen}}%
\newcommand{\tempskip}{\vspace{0.5ex}}
\begin{proof}[Step 1{\rm :}]
 Take a real number $B\not\in \pi\Z$ and let 
 \begin{equation}\label{eq:rho}
 \begin{aligned}
   \rho_1&=\id,\qquad
   \rho_2=\begin{pmatrix}
	   -i            &  0 \\
	   \hphantom{-}0 &  i
	  \end{pmatrix} \; ,\\
   \rho_3&=\rho_3(B)=
          \begin{pmatrix}
           q(B) & i \delta(B) \\
          i\delta(B) & \bar q(B)
           \end{pmatrix} 
      \qquad (\delta\in \R,~ q\bar q+\delta^2=1) \; ,
 \end{aligned}
 \end{equation}
 where
 \begin{equation}\label{eq:C}
     q(B) =i\cos A -\cos B \; .
 \end{equation}
 Then we have 
 \[
   \bar \rho_j \rho_j=\id\qquad (j=1,2,3) \; .
 \]
 Since $\Trig(A,B_0,\pi/2)$ is non-degenerate (by the assumption
 \ref{item:gen:angle}), 
 $A$, $B_0$ and $C=\pi/2$ satisfy the relation \eqref{eq:tri}.
 Then for each $B$ sufficiently close to $B_0$,
 there exists an abstract spherical triangle with angles  $A$, $B$ and 
 $\pi/2$.
 We identify the domain $D\subset\overline M$ with this triangle.
 Then by reflecting the metric,
 we have $d\sigma^2_{A,B,\pi/2}\in\metone(\overline M)$.
 Let $M:=\overline M\setminus\{p_1,\dots,p_n\}$ and
 $\pi\colon{}\widetilde M\to M$ the universal covering.
 By Proposition~\ref{prop:A} in the appendix 
 of this paper, the developing map 
 $\hat g_{A,B,\pi/2}$ of $d\sigma^2_{A,B,\pi/2}$ is defined on
 $\widetilde M$. 
 To simplify the notation, we set 
 \[
   \hat g_B:=\hat g_{A,B,\pi/2}\colon{}
        \widetilde M \longrightarrow \C\cup \{\infty\} \; . 
 \]
 Then by the monodromy principle (Lemma~\ref{lem:p-I}), we have
 \[
      \overline{\hat g_B\circ \mu_j}=\rho_j\star \hat g_B \qquad (j=1,2,3) \; .
 \]
\renewcommand{\qedsymbol}{}
\end{proof}
\begin{proof}[Step 2{\rm :}]
 One may regard the triangle $D\setminus \{V_2\}$ as generating 
 $\widetilde M$ by the three reflections across its edges.  
 We denote these reflections 
 by $\hat \mu_1$, $\hat \mu_2$, and $\hat \mu_3$; 
 that is, each $\hat \mu_j$ ($j=1,2,3$) is an antiholomorphic 
 transformation on $\widetilde M$ which preserves the $j$'th edge of 
 the triangle $D\setminus \{V_2\}$. 
 We set 
 \[
   \hat Q:=Q\circ \pi \;.
 \]
 Let $F=F_{t,B}$ be a solution of the following ordinary differential
 equation on $\widetilde M$:  
 \begin{equation}\label{eq:our_ode}
    F^{-1}dF=  
     t\begin{pmatrix} 
       {\hat g}_B & -{\hat g}_B^2 \\
       1 & -{\hat g}_B
      \end{pmatrix} \frac{\hat Q}{d{\hat g}_B},\qquad F(V_3)=\id\;.
 \end{equation}
 Such a solution $F_{t,B}$ is uniquely determined on $\widetilde M$.
 By \eqref{eq:our_ode}, the right-hand side of the ordinary
 differential equation is traceless and so $F_{t,B}$ takes values in
 $\SL(2,\C)$. 
 
 Then $\overline{F_{t,B}\circ \hat\mu_j}$ ($j=1,2,3$) has the Hopf
 differential $\hat Q=\overline{\hat Q\circ \hat \mu_j}$ and the secondary 
 Gauss map satisfies 
 $\overline{\hat g_B\circ \hat\mu_j}=\rho_j\star \hat g_B$.
 However, $F_{t,B}\,\rho_j^{-1}$ also has  the Hopf differential
 $\hat Q$ and secondary Gauss map $\rho_j\star \hat g_B$, by 
 \eqref{eq:three}.  
 Thus, by \eqref{eq:our_ode}, we have
 \[
   \left(\overline{F_{t,B}\circ \hat\mu_j}\right)^{-1}\!\!
       d \left(\overline{F_{t,B}\circ \hat\mu_j}\right)=
        \left(F_{t,B}\,\rho_j^{-1}\right)^{-1}
       d\left(F_{t,B}\,\rho_j^{-1}\right)=
   \rho_j 
   \begin{pmatrix} 
    \hat g_B & -\hat g_B^2 \\
    1 & -\hat g_B
   \end{pmatrix} \frac{\hat Q}{d\hat g_B}\,\rho_j^{-1}\;.
 \]
 This implies that $\overline{F_{t,B}\circ \hat\mu_j}$ and 
 $F_{t,B}\,\rho_j^{-1}$ are both solutions of the same ordinary
 differential equation, and thus they differ only by the choice of 
 initial values at the base point $V_3$.
 So there exists a matrix $\sigma_j(t,B)\in \SL(2,\C)$ such that
 \begin{equation}\label{eq:sigma}
  \overline{F_{t,B}\circ \hat\mu_j}=\sigma_j(t,B)\, 
   F_{t,B}\, \rho_j^{-1}\qquad (j=1,2,3).
 \end{equation}
 Since $\overline{F_{t,B}\circ \hat\mu_j(V_3)}=F(V_3)=\id$
 for $j=1,2$, 
 \[
    \id=\sigma_j(t,B)\,\rho_j^{-1}\qquad (j=1,2)\;.
 \]
 holds.
 Thus we have
 \begin{equation}\label{eq:A}
  \sigma_1(t,B)=\rho_1=\id,
   \qquad \sigma_2(t,B)=\rho_2=
   \begin{pmatrix}
    -i            & 0 \\
    \hphantom{-}0 & i
   \end{pmatrix}\;.
 \end{equation}
 In particular, the matrices $\sigma_1(t,B)$ and $\sigma_2(t,B)$ do not
 depend on $t$ nor the angle $B$.  
 Since $F_{0,B}=\id$, \eqref{eq:sigma} implies that
 \begin{equation}\label{eq:q}
  \sigma_3(0,B)=\rho_3(B)
    =\begin{pmatrix}
       q(B) & i \delta(B) \\
       i\delta(B) & \overline{q(B)}
   \end{pmatrix}\;. 
 \end{equation}
\renewcommand{\qedsymbol}{}
\end{proof}
\begin{proof}[Step 3{\rm :}]
 Next we shall describe the matrix $\sigma_3(t,B)$.
 We have
 \begin{align*}
    F_{t,B}&=F_{t,B}\circ \hat\mu_3 \circ \hat\mu_3 
            =\overline{%
               \overline{F_{t,B}\circ \hat\mu_3 \circ \hat\mu_3}}\\
           &=\overline{%
               \sigma_3(t,B)\bigl(F_{t,B}\circ \hat\mu_3\bigr)
                   \rho_3(B)^{-1}}
           =\overline{%
               \sigma_3(t,B)}\,\sigma_3(t,B)\,F_{t,B}\,\rho_3(B)^{-1}
               \overline{\rho_3(B)^{-1}}\\
           &=\overline{\sigma_3(t,B)}\,\sigma_3(t,B)\,F_{t,B}\;,
 \end{align*}
 where we used the fact $\overline{\rho_3(B)}\rho_3(B)$ is the
 identity. 
 Thus we have
 \[
   \overline{\sigma_3(t,B)}\sigma_3(t,B)=\id\;.
 \]
 By Lemma~\ref{fact}, the matrix $\sigma_3(t,B)$ can be written in  the
 following form
 \begin{equation}\label{eq:form}
  \sigma_3(t,B)=
   \begin{pmatrix}
       p(t,B) & i \nu_1(t,B) \\
       i\nu_2(t,B) & \overline{p(t,B)}
   \end{pmatrix}
      \qquad \text{with } 
      \nu_1,\nu_2\in \R \; , \quad 
      p\bar p+\nu_1\nu_2=1 \; .  
 \end{equation}
 We also have 
 \begin{align*}
  F_{t,B}\circ \hat\mu_2 \circ \hat\mu_3 
     &=\overline{\overline{F_{t,B}\circ \hat\mu_2 \circ \hat\mu_3}}
      =\overline{\sigma_3(t,B)\bigl(F_{t,B}\circ \hat\mu_2\bigr)\rho_3(B)^{-1}}\\
     &=
       \overline{\sigma_3(t,B)}\,\sigma_2(t,B)\,F_{t,B}\,\rho_2(B)^{-1}\,
       \overline{\rho_3(B)^{-1}}\;. 
 \end{align*}
 We may assume that the small disk centered at $V_1$ consists of
 $2l$-copies of $D$.
 Let $b$ be the branching order of $g$ at $V_1$.
 By the condition \ref{item:gen:branch}, we have
 \begin{equation}\label{eq:ang-rel}
    A=\pi \frac{b+1}{l} \; .
 \end{equation}
 Since $\hat\mu_2 \circ \hat\mu_3$ is the rotation of angle $2A$ at
 $V_1$,  
 we have $(\hat\mu_2 \circ \hat\mu_3)^l=\id$:
 \[
   F_{t,B}=F_{t,B}\circ (\hat\mu_2 \circ \hat\mu_3)^l
   =
   \left(\overline{\sigma_3(t,B)}\,\sigma_2(t,B)\right)^l
   F_{t,B}
   \left(\overline{\rho_3(B)}\,\rho_2(B)\right)^{-l}.
 \]
 On the other hand, one can easily check that the eigenvalues of
 $\overline{\rho_3(B)}\rho_2(B)$ are $\{-e^{iA},-e^{-iA}\}$. 
 Then by \eqref{eq:ang-rel}, the eigenvalues of
 $\left(\overline{\rho_3(B)}\rho_2(B)\right)^{l}$ are $\{\pm 1,\pm 1\}$, 
 that is
 \[
   \left(\overline{\rho_3(B)}\rho_2(B)\right)^{l}=\pm\id\;.
 \]
 So we have
 \[
   F_{t,B}=\pm \left(\overline{\sigma_3(t,B)}\sigma_2(t,B)\right)^l
   F_{t,B}
 \]
 which implies that
 \[
   \left(\overline{\sigma_3(t,B)}\sigma_2(t,B)\right)^l=\pm \id\;.
 \]
 This implies that the eigenvalues of
 $\overline{\sigma_3(t,B)}\sigma_2(t,B)$ are of the form
 $\{e^{\pi i N/l},e^{-\pi i N/l}\}$ for some integer $N$.
 Since $\overline{\sigma_3(t,B)}\sigma_2(t,B)$ is continuous with
 respect to the parameter $t$, 
 we can conclude that the eigenvalues of
 $\overline{\sigma_3(t,B)}\sigma_2(t,B)$ do not change by $t$. 
 In particular, 
 \begin{equation}\label{eq:trace1}
  \trace\overline{\sigma_3(t,B)}\sigma_2(t,B)
   =\trace\overline{\sigma_3(0,B)}\sigma_2(0,B)\;.
 \end{equation}
 On the other hand, since 
 $F_{t,B}\circ \hat\mu_2 \circ \hat\mu_3 
 =\overline{\sigma_3(t,B)}\sigma_2(t,B)\,F_{t,B}\,
 (\overline{\rho_3(B)}\rho_2(B))^{-1}$ and $F_{0,B}=\id$,
 we have 
 \begin{equation}\label{eq:trace2}
  \overline{\sigma_3(0,B)}\sigma_2(0,B)=
   \overline{\rho_3(B)}\rho_2(B)\;.
 \end{equation}
 By \eqref{eq:trace1}, \eqref{eq:trace2} and \eqref{eq:form}, we have
 \[
  2 \Im p(t,B)=\trace\overline{\sigma_3(t,B)}\,\sigma_2(t,B)
   =\trace\overline{\rho_3(B)}\,\rho_2(B)
   =2 \cos A\;.
 \]
\renewcommand{\qedsymbol}{}
\end{proof}
\begin{proof}[Step 4{\rm :}]
 Since $\sigma_3(0,B)=\rho_3(B)$, we have
 \[
   \Re p(0,B_0)=-\cos B_0\;.
 \]
 Now we would like to find a real valued smooth function $B(t)$ such
 that 
 \[
  \Re p(t,B(t))=-\cos B_0 \qquad (B(0)=B_0)\;.
 \]
 By the implicit function theorem, a sufficient condition for 
 the existence of such a $B(t)$ is 
 \[
   \left.
     \frac{\partial \Re p(t,B)}{\partial B}
            \right |_{(t,B)=(0,B_0)} \ne 0\;,
 \]
 and by \eqref{eq:form}, \eqref{eq:q}, \eqref{eq:C} and the assumption
 \ref{item:gen:angle}, we have 
 \begin{align*}
  \left.
  \frac{\partial \Re p(t,B)}{\partial B}
  \right |_{(t,B)=(0,B_0)}
  &=\left.
  \frac{\partial \Re p(0,B)}{\partial B}
  \right |_{B=B_0} \\
  &=\left.
  \frac{\partial \Re q(B)}{\partial B}
  \right |_{B=B_0} \\
  &=\left.
  -\frac{\partial \cos B}{\partial B}
  \right |_{B=B_0}=\sin B_0\ne 0\;. 
 \end{align*}
 This proves the existence of such a $B(t)$ 
 $(|t|<\varepsilon)$ for a sufficiently 
 small $\varepsilon>0$.
\renewcommand{\qedsymbol}{}
\end{proof}
\begin{proof}[Step 5{\rm :}]
 When $t=0$, it holds that $\sigma_3(t,B)=\rho_3(B)$, so 
 $\nu_1\nu_2>0$
 for sufficiently small $t$ ($|t|<\varepsilon$), by continuity.
 Now we set
  \[
     F_{t}:=
          \begin{pmatrix}   
            u(t)   &   0 \\
            0   &   u(t)^{-1}
          \end{pmatrix}
          F_{t,B(t)} \; , 
          \qquad
         u(t)=\sqrt[4]{\frac{\nu_2(t,B(t))}{\nu_1(t,B(t))}}
         \;.
  \]
 Obviously, $F_t$ satisfies the ordinary differential equation
 \eqref{eq:our_ode} for $B=B(t)$.
 In particular, $F_t$ has the Hopf differential $t\hat Q$ and the
 secondary Gauss map $\hat g_{B(t)}$. 
 Then by \eqref{eq:A} and \eqref{eq:form}, we get
 the following relations 
 \[
   \overline{F_t \circ \mu_j}
       =\sigma_j(t)\, F_t\, \rho_j(B(t)) \qquad (j=1,2,3)\;,
 \]
 where
 \begin{align*}
   \sigma_1(t)&=\id,\qquad
   \sigma_2(t)=\rho_2=
  \begin{pmatrix}
   -i &            0\\
   \hphantom{-}0 & i
  \end{pmatrix}\\
  \intertext{and}
  \sigma_3(t)&=
  \begin{pmatrix}
   p & i \sqrt{\nu_1\nu_2} \\
   i\sqrt{\nu_1\nu_2} & \bar p 
  \end{pmatrix}\qquad
  \bigl(
      p = p(t,B(t)),~\nu_j=\nu_j(t,B(t)),~j=1,2
  \bigr)\;.
 \end{align*}
 Since $\Im p(t,B(t))=\cos A$ and $\Re p(t,B(t))=-\cos B_0$, we have
 \[
    p(t,B(t))=i\cos A-\cos B_0=q(B_0) \;.
 \]
 Thus we have $\sigma_j(t)=\rho_j(B_0)$ for $j=1,2,3$.
 We now set
 \[
      f_t:=F_tF_t^*\colon\widetilde M\longrightarrow H^3\;.
 \]
 By \eqref{eq:first}, the first fundamental form of $f_t$ is given by
 \[
   ds^2:=\left(1+|\hat g_{B(t)}|^2\right)^2 
   \left|\frac{\hat Q}{d\hat g_{B(t)}}\right|^2\;.
 \]
 By the condition \ref{item:gen:branch} of the theorem, it is positive
 definite,
 and thus $f_t$ is a  conformal \cmcone{} immersion whose
 Hopf differential is $t\hat Q$ and the secondary Gauss map 
 $\hat g_{B(t)}$. 
\renewcommand{\qedsymbol}{}
\end{proof}
\begin{proof}[Step 6{\rm :}]
 We shall now prove that the conformal \cmcone{} immersion $f_t$
 is single-valued on $M=\overline M\setminus\{p_1,\dots,p_n\}$.
 Let $s$ be a non-negative integer and
 $\hat \mu_{i_1},\hat \mu_{i_2}, \dots , \hat \mu_{i_{2s}}$
 are sequences of three reflections $\hat \mu_1,\hat \mu_2,\hat \mu_3$
 on $M$ such that
 \[
   \pi\circ \hat \mu_{i_1}\circ \hat \mu_{i_2}\circ \cdots \circ 
   \hat \mu_{i_{2s}}=\pi\;,
 \]
 where $\pi\colon{}\widetilde M\to M$ be the universal covering.
 To show the single-valued property of $f_t$, it is sufficient to show
 that 
 $f_t= 
 f_t\circ \hat\mu_{i_1}\circ 
          \hat\mu_{i_2}\circ \cdots\circ \hat\mu_{i_{2s}}$.
 In fact, we have
 \begin{multline*}
  F_t\circ \hat\mu_{i_1}\circ 
  \hat\mu_{i_2}\circ \dots \circ \hat\mu_{i_{2s}}
  = \left(\overline{\rho_{i_1}(B_0)}\rho_{i_2}(B_0)
  \dots \overline{\rho_{i_{2s-1}}(B_0)}\rho_{i_{2s}}(B_0)\right)
  F_t \\
  \left(\overline{\rho_{i_1}(B(t))}\rho_{i_2}(B(t))
  \dots \overline{\rho_{i_{2s-1}}(B(t))}\rho_{i_{2s}}(B(t))\right)\;.
 \end{multline*}
 Since $g$ is single-valued on $M$, $\hat g:=g\circ\pi$ satisfies
 \[
   \hat g=\hat g\circ \hat\mu_{i_1}\circ \cdots \circ \hat\mu_{i_{2s}}\;.
 \]
 On the other hand, since $\hat g$ is the secondary Gauss map of
 $F_{t,B_0}$,
 we have by \eqref{eq:three} that
 \[
   \hat g=
   \hat g\circ \hat\mu_{i_1}\circ \cdots \circ \hat\mu_{i_{2s}}
   = \left(\overline{\rho_{i_1}(B_0)}\,\rho_{i_2}(B_0)
     \dots \overline{\rho_{i_{2s-1}}(B_0)}\,\rho_{i_{2s}}(B_0)\right)\star g\;.
  \]
 Thus we can conclude that
 \[
    \overline{\rho_{i_1}(B_0)}\,\rho_{i_2}(B_0)
    \dots \overline{\rho_{i_{2s-1}}(B_0)}\,\rho_{i_{2s}}(B_0)
     =\pm \id\;.
 \]
 Since $\rho_{i_{j}}(B(t))\in \SU(2)$ ($j=1,\dots,2s$), this implies that
 $f_t=F_tF_t^*$ is single-valued on $M$.

 Moreover, by the assumption \ref{item:gen:ord} and
 Lemma~\ref{lem:irreg}, each end is irregular.
 The Hopf differential of $f_t$ is $Q_t$, and the pseudometric
 $d\sigma^2_{f_t}$ defined in \eqref{eq:g-metric} is
 $d\sigma^2_{A,B(t),\pi/2}$.
 Since they are symmetric with respect to $D$, \eqref{eq:gauss} and 
 \eqref{eq:first} imply that the first and second fundamental forms 
 of $f_t$ are invariant under the reflections $\mu_j$ $(j=1,2,3)$.
 Then by the fundamental theorem of surfaces, $f_t$ is symmetric with
 respect to $D$.
\end{proof}
\begin{proof}[Proof of Theorem~\ref{thm:irred}]
 Let $\tau$ be a loop surrounding the point $V_2$ with the base point
 $V_3$
 and $T$ the covering transformation of $\widetilde M$ corresponding
 to $\tau$.
 Suppose that a neighborhood of $V_2$ is generated by $2k$-copies of 
 $D$.
 Then 
 \[
    T := (\hat \mu_3\circ\hat \mu_1) ^ k
 \]
 holds.
 Thus we have
 \begin{align*}
    F_{t,B}\circ\tau &=
              F_{t,B}\circ(\hat\mu_3\circ\hat\mu_1)^k\\
            &=\left(\overline{\sigma_3(t,B)}\,\sigma_1\right)^k
                F_{t,B}
             \left(\rho_1^{-1}\overline{\rho_3(B)^{-1}}\right)^k
            = \overline{\sigma_3(t,B)}^k F_{t,B}\,\overline{\rho_3(B)}^{-k}\;.
 \end{align*}
 Here, by the argument of Step 6 of the proof of Theorem~\ref{thm:gen},
 we have
 $\rho_3(B)^k=\pm\id$.
 Hence at the base point $V_3$, 
 \begin{align}\label{eq:tdiff}
    \left.\frac{\partial F_{t,B}\circ \tau}{\partial t}
       \right|_{(t,B)=(0,B_0)}
    &=
    \left.\frac{\partial}{\partial t}
       \overline{\sigma_3(t,B)}^k F_{t,B}\,\overline{\rho_3(B)}^{-k}
       \right|_{(t,B)=(0,B_0)}\\
    &=\pm\left.\frac{\partial}{\partial t}
       \overline{\sigma_3(t,B)}^k
       \right|_{(t,B)=(0,B_0)}
    \nonumber
 \end{align}
 since $F_{t,B}(V_3)=\id$.
 
 Since $F_{t,B}$ is a solution of
 \eqref{eq:our_ode}, it holds that 
 \begin{align*}
    \frac{\partial}{\partial z}
    \left[\left.\frac{\partial F_{t,B}}{\partial t}
    \right|_{(t,B)=(0,B_0)}\right]\,dz
      &= 
    d
    \left[\left.\frac{\partial F_{t,B}}{\partial t}
    \right|_{(t,B)=(0,B_0)}\right]\\
      &=
    \left.\frac{\partial }{\partial t} dF_{t,B}\right|_{(t,B)=(0,B_0)}\\
      &=
    \left.\frac{\partial}{\partial t}
        \left[
            t F_{t,B}
             \begin{pmatrix}
	       {\hat g}_B & -{\hat g}_B^2 \\
                1  & -{\hat g}_B
	     \end{pmatrix}\frac{\hat Q}{d\hat g_B}
        \right]\right|_{(t,B)=(0,B_0)} \\
      &=
         \begin{pmatrix}
	       \hat g & -\hat g^2 \\
                1  & -\hat g\hphantom{^2}
	 \end{pmatrix}\frac{\hat Q}{d\hat g} \; , 
 \end{align*}
 because $\hat g_{B_0}=\hat g=g\circ\pi$.
 Then by the assumption and \eqref{eq:tdiff}, 
 we have
 \[
    \left.\frac{\partial}{\partial t}
       \overline{\sigma_3(t,B)}^k
       \right|_{(t,B)=(0,B_0)}
       =\pm\oint_{\tau}
          \begin{pmatrix}
	       g & -g^2 \\
               1  & -g\hphantom{^2}
          \end{pmatrix}\frac{Q}{dg} \neq 0\;.
 \]
 This implies that $\sigma_3(t)=\sigma_3(t,B(t))$ is not constant 
 on $\{t\,;\,|t|<\varepsilon\}$ for sufficiently small $\varepsilon>0$, 
 and hence we have 
 \begin{equation}\label{eq:c-irred}
  B(t)\neq B_0 \qquad \text{for}\quad  0<|t|<\varepsilon\;.
 \end{equation}
 Hence the eigenvalues of $\rho_3(t):=\rho_3(B(t))$ are different from
 those of $\rho_3(B_0)$.
 
 The secondary Gauss map $\hat g_t$ of $F_t$ changes by the covering
 transformation $T$ as
 \[
    \hat g_t\circ T = \hat g_t\circ(\hat\mu_3\circ\hat\mu_1)^k
                 = \overline{\rho_3(t)}^k\star \hat g_t\;.
 \]
 Now, let $\widetilde V_2=\hat\mu_2(V_2)$ 
 (see Figure~\ref{fig:endreflect}).
\begin{figure}
\setlength{\unitlength}{5mm}
\small
\begin{picture}(11,7)(-5,-1)
 \Thicklines
  \put(0,4){\line(1,-1){4}}
  \put(0,0){\line(1,0){4}}
  \put(0,0){\line(0,1){4}}
 \thinlines
  \put(-4,0){\line(1,1){4}}
  \put(-4,0){\line(1,0){4}}
 \put(0.2,-0.2){\makebox(0,0)[lt]{$V_3$}}
 \put(4,-0.2){\makebox(0,0)[lt]{$V_2$}}
 \put(-0.2,4){\makebox(0,0)[rb]{$V_1$}}
 \put(-4,-0.2){\makebox(0,0)[lt]{$\widetilde V_2$}}
 \put(2,0){\vector(0,1){0.5}}
 \put(2,0){\vector(0,-1){0.5}}
 \put(0,2){\vector(1,0){0.5}}
 \put(0,2){\vector(-1,0){0.5}}
 \put(2,2){\vector(1,1){0.35355339}}
 \put(2,2){\vector(-1,-1){0.35355339}}
 \put(-2,2){\vector(-1,1){0.35355339}}
 \put(-2,2){\vector(1,-1){0.35355339}}
 \put(2.2,-0.3){\makebox(0,0)[lt]{$\hat\mu_1$}}
 \put(-0.3,1.8){\makebox(0,0)[rt]{$\hat\mu_2$}}
 \put(2.6,2){\makebox(0,0)[lt]{$\hat\mu_3$}}
 \put(-3.2,2){\makebox(0,0)[lt]{$\tilde\mu_3$}}
\end{picture}
\caption{$\widetilde V_2$ and $\tilde\mu_3$.}
\label{fig:endreflect}
\end{figure}
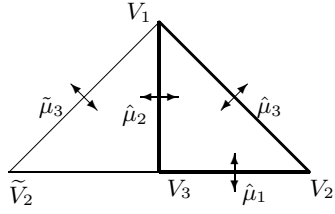
 We denote by $\tilde\mu_3$ the reflection about the edge $V_1\widetilde V_2$.
 Then we have
 \[
    \tilde\mu_3=\hat\mu_2\circ\hat\mu_3\circ\hat \mu_2\;.
 \]
 Let $\tilde \tau$ be a loop surrounding $\widetilde V_2$ with base
 point $V_3$, and let $\widetilde T$ be the covering transformation
 corresponding to $\tilde \tau$.  Then we have
 \[
    \widetilde T = (\hat\mu_1\circ\tilde\mu_3)^k=
               (\hat\mu_1\circ\hat\mu_2\circ\hat\mu_3\circ\hat\mu_2)^k\;,
 \]
 and
 \[
    \hat g_t\circ\widetilde T = 
       \left(\overline{\mathstrut\rho_1}\,\rho_2\,
                       \overline{\rho_3(t)}\,\rho_2\right)^k
          \star \hat g_t
          = (-1)^{k-1}\rho_2\left(\overline{\rho_3(t)}^k\right)
                  \rho_2\star \hat g_t\;.
 \]
 So, to show irreducibility, it is sufficient to show that the
 matrices
 \[
      a:=\overline{\rho_3(t)}^k
      \qquad \text{and}\qquad
      b:=\rho_2\overline{\rho_3(t)}^k\rho_2 = \rho_2\,a\,\rho_2 
 \]
 do not commute.
 By \eqref{eq:rho}, $\overline{\rho_3(t)^k}\rho_3(t)^k=\id$ holds.
 Then by Lemma~\ref{fact}, the off-diagonal components of $a$ are
 coincide and pure imaginary .
 Set
 \[
     a=\overline{\rho_3(t)}^k = \begin{pmatrix}
				  r & i\beta \\
                                  i\beta & \bar r
                                 \end{pmatrix}
     \qquad (\beta\in\R,\,r\bar r+\beta^2=1)\;.
 \]
 Then we have
 \[
     b = \rho_2 \, a \,\rho_2 =
        \begin{pmatrix}
	    -r & i\beta \\
            i\beta & -\bar r
        \end{pmatrix}\;,
 \]
 and 
 \[
   \left[a,b\right] = a b - b a
       =\begin{pmatrix}
	   0 & -2\beta\Im r\\
          2\beta\Im r & 0
        \end{pmatrix}\;,
 \]
 that is,  $a$ and $b$ commute if and only if $\beta=0$ or $r$ is a
 real number.

 First, we consider the case $\beta=0$.
 Then $a=\overline{\rho_3(t)}^k$ is a diagonal matrix whose eigenvalues
 are different from $\pm 1$ for sufficiently small $t$, because $B(t)$ is
 not constant.
 In particular, the two eigenvalues of $a$ are distinct.
 This implies that the eigenspaces of $a$ coincide of those of 
 $\overline{\rho_3(t)}$.
 Since $a$ is diagonal, this implies that $\rho_3(t)$ is also a diagonal
 matrix, a contradiction.
 
 Next, assume $r$ is real.
 Then there exists a real number $\theta$ such that 
 \begin{multline*}
      a = \overline{\rho_3(t)}^k
        =\begin{pmatrix}
	  \cos\theta & i\sin\theta \\
          i\sin\theta & \cos\theta
         \end{pmatrix}
        =P\begin{pmatrix} e^{i\theta} & 0 \\ 0 & e^{-i\theta}\end{pmatrix}
         P^{-1}\;,\\
      \text{where }
         P=\frac{1}{\sqrt{2}}
             \begin{pmatrix}
                   1 & -1 \\
                   1 & \hphantom{-}1
             \end{pmatrix}\; \qquad \text{and}\qquad
          \theta\in\R\setminus\pi\Z\;.
 \end{multline*}
 In this case, $(P^{-1}\overline{\rho_3(t)}P)^k$ is a diagonal matrix 
 whose eigenvalues differ from $\pm 1$, for sufficiently small $t\neq 0$.
 Then, by a similar argument to the previous case, we have
 $P^{-1}\overline{\rho_3}(t)P$ is diagonal.
 If $A\not\equiv\pi/2 \pmod{\pi}$, this contradicts to \eqref{eq:rho}
 and \eqref{eq:C}.
 Hence $a$ and $b$ do not commute.
\end{proof}
\section{An example of genus one}
\label{sec:torus}
In the final section, we construct an example of a \cmcone{} surface 
of genus one with four irregular ends.

Let $\Gamma=\Z\oplus i\Z$ be the lattice of Gaussian integers of 
$\C$ and let 
\[
   \overline M:= \C/2\Gamma\;.
\]
We consider $\overline M$ as the square 
$[-\frac{1}{2},\frac{3}{2}]\times[-\frac{1}{2},\frac{3}{2}]$
in $\C=\R^2$, with opposite edges identified.  
Take a triangle $D$ on $\overline M$ as in Figure~\ref{fig:fund-torus}.
Then $\overline M$ is obtained from $D$ by reflecting across the 
edges of $D$.  
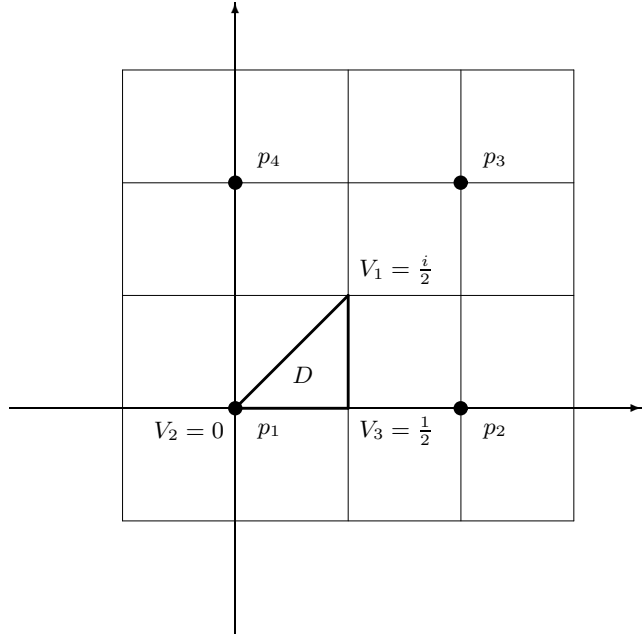
\begin{figure}
 \setlength{\unitlength}{30mm}
 \small
 \begin{center}
  \begin{picture}(3,3)(-1,-1)
   \put(-1,0){\vector(1,0){2.8}}
   \put(0,-1){\vector(0,1){2.8}}
   \drawline(-0.5,-0.5)(-0.5,1.5)(1.5,1.5)(1.5,-0.5)(-0.5,-0.5)
   \drawline(-0.5,1)(1.5,1)
   \drawline(1,-0.5)(1,1.5)
   \drawline(-0.5,0.5)(1.5,0.5)
   \drawline(0.5,-0.5)(0.5,1.5)
   \Thicklines
   \drawline(0,0)(0.5,0)(0.5,0.5)(0,0)
   \put(0,0){\circle*{0.05}}
   \put(1,0){\circle*{0.05}}
   \put(0,1){\circle*{0.05}}
   \put(1,1){\circle*{0.05}}
   \put(0.3,0.15){\makebox(0,0){$D$}}
   \put(0.55,-0.1){\makebox(0,0)[lc]{$V_3=\frac{1}{2}$}}
   \put(0.55,0.55){\makebox(0,0)[lb]{$V_1=\frac{i}{2}$}}
   \put(-0.05,-0.1){\makebox(0,0)[rc]{$V_2=0$}}
   \put(0.1,-0.1){\makebox(0,0)[lc]{$p_1$}}
   \put(1.1,-0.1){\makebox(0,0)[lc]{$p_2$}}
   \put(1.1,1.1){\makebox(0,0)[lc]{$p_3$}}
   \put(0.1,1.1){\makebox(0,0)[lc]{$p_4$}}
  \end{picture} 
 \end{center}
\caption{Fundamental domain of the torus}
\label{fig:fund-torus}
\end{figure}

Using the Weierstrass $\wp$ function with respect to $\Gamma$
(not with respect to $2\Gamma$), 
we set
\[
    Q = \bigl(\wp'(z)\bigr)^2\,dz^2\;.
\]
Then $Q$ has poles at $\{p_1,p_2,p_3,p_4\}=\{0,1,1+i,i\}$, each 
with order $6$.
The $\wp$-function with respect to the square lattice has the following
properties
\[
   \overline{\wp(\bar z)}=\wp(z)\;,\qquad
   \overline{\wp(-\bar z)}=\wp(z)\;,\qquad
   \overline{\wp(i\bar z)}=-\wp(z)\;.
\]
Hence $Q$ is symmetric with respect to $D$.

Consider an abstract spherical triangle $\Trig(A,B_0,C)$ 
with 
\[
   A = \frac{3}{4}\pi\;,\qquad
   B_0 = \frac{\pi}{2}\;,\qquad
   C = \frac{3}{2}\pi\;,
\]
and identify the triangle with the fundamental region $D$.
Then the metric of $\Trig(A,B_0,C)$ can be extended to 
$d\sigma^2_{A,B_0,C}\in\metone(\overline M)$ by reflections.
Since $A$, $B_0$ and $C$ satisfy \eqref{eq:tri}, $d\sigma^2_{A,B_0,C}$
is non-degenerate.
Let $g$ be the developing map of $d\sigma^2_{A,B_0,C}$.
Since the conical orders of $d\sigma^2_{A,B_0,C}$ are integers, $g$ is
well-defined on $\C$.

Now, we prove that $g$ is well-defined on $\overline M$.
By the monodromy principle (Lemma \ref{lem:p-I}), one can choose 
$g$ such that $\overline{g\circ\hat\mu_j}=\rho_j\star 
g$ ($j=1,2,3$), where
\[
   \rho_1:=\id,\qquad
   \rho_2:=\begin{pmatrix}
              i &  \hphantom{-}0 \\
              0 &  -i
	   \end{pmatrix},\qquad
   \rho_3:=\frac{1}{\sqrt{2}}
           \begin{pmatrix}
	    \hphantom{-}i & \pm i \\
            \pm i         & -i
           \end{pmatrix}
\]
and $\hat\mu_1$, $\hat\mu_2$ and $\hat\mu_3$ are reflections along 
the edges $V_3V_1$, $V_1V_2$ and $V_2V_3$, respectively.
We denote by  $\tau_1$ and $\tau_2$ the translations $z\mapsto z+1$ 
and $z\mapsto z+i$ respectively.
Then 
\[
    \tau_1=\hat\mu_2\circ\hat\mu_3\circ\hat\mu_1\circ\hat\mu_3,
    \qquad
    \tau_2=\hat\mu_3\circ\hat\mu_2\circ\hat\mu_3\circ\hat\mu_1
\]
holds.
So we have
\begin{align*}
   g\circ\tau_1 &= \overline{\rho_2}\,\rho_3\,\overline{\rho_1}\,\rho_3
                  \star g
                = \rho_2\star g\;,\\
   g\circ\tau_2 &= \overline{\rho_3}\,\rho_2\,\overline{\rho_3}\,\rho_1
                  \star g
                = \pm\begin{pmatrix}
		        0 & i \\ i & 0
                     \end{pmatrix}\star g\;.
\end{align*}
Thus
\[
    g(z+2)=g\circ\tau_1\circ\tau_1(z) = g(z)\;,\qquad
    g(z+2i)=g\circ\tau_2\circ\tau_2(z) =g(z)
\]
hold.
This shows that $g$ is invariant under the action of the lattice
$2\Gamma$.
Hence $g$ is a meromorphic function on $\overline M$.

One can easily see that the same result as Theorem~\ref{thm:gen}
holds when $C=3\pi/2$, instead of $\pi/2$.  Hence 
we have a one-parameter family $\{f_t\}$ of \cmcone{} immersions of
$\overline M\setminus\{p_1,p_2,p_3,p_4\}$ into $H^3$ with irregular
ends.

\appendix
\begingroup
\renewcommand{\thesection}{\Alph{section}}
\section{}
\label{app:A}
\endgroup
For a compact Riemann surface $\overline M$ and points 
$p_1,\dots,p_n\in \overline M$, a conformal metric $d\sigma^2$ of
constant curvature $1$ on $M:=\overline M\setminus\{p_1,\dots,p_n\}$ is
an element of $\metone(\overline M)$ if there exist real numbers
$\beta_1,\dots,\beta_n >-1$ so that each $p_j$ is a conical singularity
of conical order $\beta_j$, that is, if $d\sigma^2$ is 
asymptotic to $c_j|z-p_j|^{2\beta_j}\,dz\cdot d\bar z$ 
at $p_j$, for $c_j\neq 0$ and $z$ a local complex coordinate around
$p_j$.  
We call the formal sum 
\begin{equation}\label{eq:divisor}
   D:=\sum_{j=1}^n \beta_j\, p_j 
\end{equation}
the {\em divisor\/} corresponding to $d\sigma^2$.
For a pseudometric $d\sigma^2\in\metone(\overline M)$ with divisor $D$,
there is a holomorphic map $g\colon{}\widetilde M\to\CP^1$ such that
$d\sigma^2$ is the pull-back of the Fubini-Study metric of $\CP^1$.  
This map, called the {\em developing map\/} of $d\sigma^2$, is uniquely
determined up to M\"obius transformations $g \mapsto a\star g$ for 
$a \in \SU(2)$.  
We have the following expression
\[
      \pi^*d\sigma^2=\frac{4\,dg\cdot d\bar g}{(1+|g|^2)^2}\;, 
\]
where $\pi:\widetilde M\to M$ is a covering projection.

Consider $d\sigma^2\in\metone(\overline M)$ with divisor $D$ as in 
\eqref{eq:divisor} and with the developing map $g$.  
Since the pull-back of the 
Fubini-Study metric of $\CP^1$ by $g$ is invariant under the deck
transformation group $\pi_1(M)$ of 
$M:=\overline M\setminus\{p_1,\dots,p_n\}$, there is a representation 
\[
    \rho_g\colon{}\pi_1(M)\longrightarrow 
    \SU(2) 
\]
such that 
\[ 
    g\circ T^{-1}  = \rho_g(T)\star g \qquad (T\in\pi_1(M))\;.
\]
The metric $d\sigma^2$ is called {\em reducible\/} if the image 
of $\rho_g$ is a commutative subgroup in $\SU(2)$, and is called 
{\em irreducible\/} otherwise.  
Since the maximal abelian subgroup of $\SU(2)$ is $\U(1)$, 
the image of $\rho_g$ for a reducible $d\sigma^2$ lies in a subgroup 
conjugate to $\U(1)$, and this image might be simply the identity.  
We call a reducible metric $d\sigma^2$ {\em $\Hyp^3$-reducible\/}
if the image of $\rho_g$ is the identity, and {\em $\Hyp^1$-reducible\/}
otherwise (for more on this, see \cite[Section~3]{ruy1}).

The following assertion was needed in Section~\ref{sec:reflection}: 
\begin{proposition}\label{prop:A}
Let  $d\sigma^2$ be a metric of constant curvature $1$ defined on
$M=\overline{M} 
\setminus \{p_1,\dots,p_n\}$ whose conical order at each $p_j$ is an 
integer. 
Then the developing map $g$ of $d\sigma^2$ is single-valued 
on the universal covering of $M$.
\end{proposition}
Let $p_1,\dots,p_{n-1}$ be distinct points in $\C$ and $p_n=\infty$.  
We set 
\[
   M_{p_1,\dots,p_n}:=\C\cup\{\infty\} 
     \setminus \{p_1,p_2,\dots,p_n\} \qquad (p_n=\infty)
\]
and $\widetilde{M}_{p_1,\dots,p_n}$ its universal covering.  
\begin{corollary}\label{cor:A}
 Let $d\sigma^2$ be a metric of constant curvature $1$ defined on
 $M_{p_1,\dots,p_n}$ $(p_n=\infty)$ whose conical order at each 
 $p_j$ is an integer. 
 Then the developing map $g$ of $d\sigma^2$ is  single-valued  on
 $M_{p_1,\dots,p_n}$ and extends as a meromorphic function on
 $\C\cup\{\infty\}$.
 Moreover, the divisor of $d\sigma^2$ coincides with the ramification
 divisor of the meromorphic function $g$.
\end{corollary}


\end{document}